\documentclass[a4paper]{gtart}

\usepackage[inner=20mm, outer=20mm, textheight=245mm]{geometry}

\usepackage{psfrag, graphicx, subcaption, epsfig, color}

\usepackage{amsmath, amssymb, latexsym, euscript, mathtools}

\usepackage[colorlinks,citecolor=blue,linkcolor=blue, backref=page]{hyperref}
\usepackage[nameinlink]{cleveref}

\usepackage[margin=1cm]{caption}

\usepackage{mathptmx, wrapfig}

\usepackage{algorithmicx}
\usepackage{algpseudocode}

\usepackage{longtable}
\usepackage{enumitem}

\usepackage{tikz, tikz-cd}
\usetikzlibrary{knots, positioning, decorations.markings, matrix}

\newcommand{\lr}[1]{%
    \left( #1 \right)%
}

\def\Vol{\operatorname{Vol}}

\def\ang{\CMcal{A}}
\def\tri{\mathcal{T}}

\def\edge{\mathfrak{e}}

\def\Z{\mathbb{Z}}
\def\R{\mathbb{R}}
\def\N{\mathbb{N}}

\newcommand{\angles}[3]{%
    \lr{#1,\; #2,\; #3}%
}

\newcommand{\angT}[1]{%
    \Theta^{(#1)}%
}

\theoremstyle{plain}
\newtheorem{theorem}{Theorem}
\newtheorem*{theorem*}{Theorem}
\newtheorem{lemma}[theorem]{Lemma}

\newtheorem{corollary}[theorem]{Corollary}

\newtheorem*{claim*}{Claim}

\theoremstyle{definition}
\newtheorem{definition}[theorem]{Definition}
\newtheorem*{definition*}{Definition}

\newtheorem{conjecture}[theorem]{Conjecture}

\theoremstyle{remark}
\newtheorem{remark}[theorem]{Remark}

\numberwithin{equation}{section}

\date{\today}

\begin{document}
\title{\Large On the complexity of 2-bridge link complements}
\author{James Morgan and Jonathan Spreer}

\begin{abstract}
	We reprove a necessary condition for the Sakuma-Weeks triangulation of a 2-bridge link complement to be minimal in terms of the mapping class describing its alternating 4-string braid construction. For the 2-bridge links satisfying this condition we construct explicit angle structures on the Sakuma-Weeks triangulations and compute both multiplicative and additive lower bounds on the complexity of the link complements via volume estimates. These lower bounds are an improvement on existing volume estimates for the 2-bridge links examined.
\end{abstract}


\makeshorttitle

\section{Introduction}

This paper contains results about the complexity of $2$-bridge link complements.
These are obtained through a detailed study of triangulations of them due to Sakuma and 
Weeks \cite{sakuma_examples_1995}. Before we present our results, we first motivate in
more general terms, why we are interested in studying particularly well-structured 
triangulations of manifolds.

\paragraph*{Motivation.}
When studying a manifold $M$, a useful first step is to decompose it into simplices.
The resulting {\em triangulation} is a finite list of simplices and gluing maps 
representing the manifold. This list can then be used as input for mathematical software 
to compute topological or geometric properties of $M$.

The price we pay for using triangulations is that any given manifold can 
be triangulated in infinitely many ways. Depending on the application, some of 
these triangulations will be better choices than others, leading to the question of what are 
``good'' and what are ``bad'' triangulations.

For $M$ the $2$-dimensional sphere this can be illustrated as follows: the three 
simplicial platonic solids -- the tetrahedron, the octahedron, and the icosahedron -- have
as boundaries the only three triangulations of the $2$-sphere with the maximum number of
symmetries. Hence, they are natural choices for the ``best'' triangulation of the 
$2$-sphere. The tetrahedron is the simplicial triangulation with the smallest number of 
pieces, the octahedron is the smallest triangulation with central symmetry, and the 
icosahedron is the smallest triangulation covering a simplicial triangulation of the real 
projective plane.

In this article, we focus on the class of {\em cusped hyperbolic $3$-dimensional manifolds},
that is, orientable non-compact $3$-manifolds admitting a complete hyperbolic 
metric. These manifolds occur as the interior of compact $3$-manifolds with boundary
a finite, non-empty collection of tori. For most metrics, these are the generic 
pieces in the decomposition of an arbitrary $3$-manifold into geometric pieces, and 
hence a very important class of manifolds, see for instance \cite{maher10-random-heegaard}. The most efficient way 
to represent cusped $3$-manifolds (or, to some extent, their compact counterparts) is by 
	{\em ideal triangulations}. These are triangulations with one vertex per cusp, and no further
vertices in their interior. The represented manifold is homeomorphic to the underlying space of the 
triangulation with its vertices removed.\footnote{Note that ideal triangulations do not require the
	underlying manifold to admit a hyperbolic structure. They are also used to represent other
	manifolds, such as non-hyperbolic knot complements.}

But even among ideal triangulations, there are several types of triangulations that compete
for the title of a ``good'' triangulation of a cusped hyperbolic $3$-manifold. 

\begin{itemize}
	\item In a {\bf geometric (ideal) triangulation} every (hyperbolic ideal) tetrahedron
	      can be assigned a shape of strictly positive volume, such that their gluings 
	      satisfy Thurston's gluing equations for a complete hyperbolic metric \cite{Thurston-Book-97}. 
	      Determining whether a given triangulation is geometric involves looking at its space 
	      of angle structures, see \cite{futer_angled_2011} for details. Despite a beautiful and rich theory surrounding 
	      this topic, it is still not known if every
	      cusped hyperbolic $3$-manifold admits a geometric triangulation.
	\item It follows from work of Epstein and Penner \cite{Epstein_Penner_1988}, that every once-cusped
	      hyperbolic manifold has a canonical decomposition into hyperbolic polyhedra. 
	      If these polyhedra are all tetrahedra, this {\em canonical triangulation} is also
	      a geometric one.\footnote{It is not known, if the canonical decomposition can always be
		      simplicially subdivided such that the resulting triangulation is geometric: 
		      incompatible subdivisions of faces of polyhedra may require the introduction of 
		      zero-volume tetrahedra shapes.} If two manifolds are represented by canonical
	      triangulations, solving the homeomorphism problem reduces to testing isomorphy 
	      on these triangulations.
	\item The {\em complexity $c(M)$} of a 3--manifold $M$ is the minimum number of tetrahedra
	      required in a (pseudo--simplicial) triangulation of $M$.  In the case that $M$ is closed and irreducible, this definition agrees with the complexity defined by Matveev \cite{Matveev-complexity-1990} unless the manifold is $S^3$, $\R P^3$, or $L(3,1)$. The notion of complexity as a measure of minimal triangulations is applied more broadly, including to ideal triangulations. It acts as an important organising principle when 
	      enumerating manifolds. A triangulation realising its complexity is called 
		      {\em minimal}. Minimal triangulations have many obvious advantages in computations,
	      especially since algorithmic solutions to topological problems often scale 
	      exponentially in the size of the input triangulation. Neither canonical, nor geometric
	      triangulations must be minimal. In fact, this paper contains many examples of canonical 
	      triangulations that are not minimal.\footnote{It is an interesting question to ask whether for a given
		      cusped hyperbolic manifold, at least one minimal triangulation must be geometric. 
		      Since we do not even know if geometric triangulations always exist, this question is
		      necessarily open as well.}
\end{itemize}

The complexity and the hyperbolic volume of a hyperbolic $3$-manifold (cusped or closed)
are connected through the following one-sided inequality due to Thurston \cite{thurston_geometry_1978}:

\begin{equation}
	\label{eq:volc}
	c(M) \leq  \frac{\Vol(M)}{v_3}
\end{equation}

Here, $v_3\approx 1.0149\dots$ denotes the volume of a regular ideal hyperbolic tetrahedron -- the largest 
volume of any hyperbolic tetrahedron. The inequality essentially expresses that we need
at least as many tetrahedra in a triangulation of a manifold $M$ to realise the volume
of $M$ as a sum of volumes of this number of tetrahedra.

It follows that any geometric triangulation only involving regular ideal hyperbolic 
tetrahedra is minimal. Manifolds admitting such triangulations are referred to as {\em tetrahedral} in the literature, see \cite{Fominykh-census-2016} for a census of examples.
\Cref{eq:volc} is still one of very few lower bounds on the complexity of $3$-manifolds.

\paragraph*{Related work.} In a series of papers, Jaco, Rubinstein, and Tillmann develop and refine a method to use 
the $\mathbb{Z}_2$-Thurston norm of homology classes to obtain tight complexity bounds
for infinte families of lens spaces \cite{jaco_minimal_2009},  and other spherical $3$-manifolds \cite{jaco_mathbb_2013}.
The same group of authors together with the second author extend this theory to
determine the complexity of further infinite families of manifolds \cite{jaco_2thurston_2020}, including 
canonical triangulations of once-punctured torus bundles\footnote{The monodromy-ideal triangulations of once-punctured torus bundles, first described in \cite{gueritaud_canonical_2006}, are very similar in structure to the Sakuma-Weeks triangulations of hyperbolic $2$-bridge links, the focus of this paper.} in \cite{jaco_minimal_2020}, and additional
once-cusped hyperbolic $3$-manifolds in \cite{jaco2022complexity,rubinstein2021new}. See also \cite{Nakamura-complexity-2017} for closely related 
work by Nakamura.

Complexity bounds for Dehn-fillings have been determined by Cha \cite{Cha-complexity-2016,Cha-topological-2016,Cha-complexities-2018}, and
in more special cases in \cite{jaco2022complexity}. In \cite{Lackenby19Fibred}, Lackenby and Purcell give complexity bounds
for fibred manifolds based on properties of the monodromy gluing the fibres.

The canonical and geometric triangulations of hyperbolic $2$-bridge links were first described
by Sakuma and Weeks in \cite{sakuma_examples_1995}. They have been studied in detail by Futer in the appendix to \cite{gueritaud_canonical_2006} and 
Purcell in \cite{purcell_hyperbolic_2020}. A simplified method to compute their volumes exactly is due to  
Tsvietkova, see \cite{Tsvietkova14TwoBridgeLinks}. Bounds on their complexity are contained in work by 
Ishikawa and Nemoto \cite{ishikawa_construction_2016}, where they observe that many of the Sakuma-Weeks triangulations
are not minimal by constructing a smaller upper bound on the complexity of some $2$-bridge link complements. 
They also prove that a particular infinite family of $2$-bridge link triangulations is, in fact, minimal.
In more recent work, Aribi, Gu\'eritaud and Piguet-Nakazawa describe triangulations of twist knots (knots of bridge number two) which are roughly half the size of the improved Ishikawa and Nemoto bound -- and conjectured to be minimal \cite[Conjecture 3.3]{Aribi19Twist}.

\paragraph*{Results.}
In this paper, we further study the complexity of the Sakuma-Weeks triangulations of hyperbolic $2$-bridge link complements.
We give a very fundamental argument re-proving the Ishikawa and Nemoto non-minimality result
in \Cref{thm:nonminimal}, and conjecture that all of the remaining triangulations are minimal: 

\begin{conjecture}
	\label{conj:minimal}
	Let $K(\Omega)$ be the $2$-bridge link generated by the word $\Omega = R^{a_1}L^{a_2}\cdots(R^{a_n}\mid L^{a_n})$. 
	Let $\tri = \tri(M)$ be the Sakuma-Weeks triangulation of the complements 
	$S^3\backslash K(\Omega)$. $\tri$ is minimal if and only if $a_1=a_n = 1$ and 
	$a_i\in\{1,2\}$ for $1<i<n$ with $n\geq 2$.
\end{conjecture}

Our main contribution is an application of \Cref{eq:volc} to give lower bounds for the complexity of 
$2$-bridge link complements. More precisely, for the hyperbolic $2$-bridge links covered 
by \Cref{conj:minimal} we prove the following statement.

\newtheorem*{thm:complexityBound0.8}{Theorem~\ref{thm:complexityBound0.8}}
\begin{thm:complexityBound0.8}
Let $L=L(\Omega)$ be the $2$-bridge link associated to the word $\Omega=RL^{a_1}\cdots(L^{a_n}R\mid R^{a_n}L)$ where $a_i\in\{1,2\}$ for all $1\leq i\leq n$ and $n\geq 1$. Let $M=S^3\backslash L$ and $\tri=\tri(M)$ be the Sakuma-Weeks triangulation of $M$ and let $c(M)$ denote the complexity of $M$. Then,
\begin{equation*}
	0.8|\tri| \leq c(M) \leq |\tri|
\end{equation*}
\end{thm:complexityBound0.8}

Our proof goes by explicitly describing strict angle structures on the canonical triangulations
of all of these $2$-bridge link triangulations, alongside with an evaluation of the volume functional 
on those angle structures. This yields a lower bound on their hyperbolic volume and 
hence a lower bound on complexity by \Cref{eq:volc}. Note that exact volume computations,
as given for instance in \cite{Tsvietkova14TwoBridgeLinks}, are too complicated to infer explicit volume or complexity bounds for infinite families of $2$-bridge links. Instead, a simplified
construction of strict angle structures involving very few disctinct tetrahedra shapes is necessary to keep calculations at a manageable level.

Naturally, our bounds are best for $2$-bridge links with volumes closer to ``$v_3 \,\times$ number of tetrahedra in the Sakuma Weeks triangulation''.
Specifically, we have the following corollary.

\newtheorem*{cor:additivecorollary}{Corollary~\ref{cor:additivecorollary}}
\begin{cor:additivecorollary}
Let $W = \left\{\Omega = RL^{a_1}\cdots(L^{a_n}R\mid R^{a_n}L)\mid a_i\in\{1,2\}\text{ for }1\leq i\leq n\right\}$ and for each $C\in\N$ set
\begin{equation*}
	W_C = \left\{\Omega\in W\mid a_1+\cdots+a_n = n+C,\;n\in\N,\,n\geq C\right\}.
\end{equation*}
For $\Omega\in W_C$ let $K=K(\Omega)$ be the associated $2$-bridge link. Let $M=S^3\backslash K$ and $\tri=\tri(M)$ be the Sakuma-Weeks triangulation of $M$ with $|\tri| = 2(n+1+C).$ Let $c(M)$ denote the complexity of $M$. Then,
\begin{equation*}
	2n + 1 + (0.9632C + 0.393) \leq c(M) \leq 2n + 1 + (2C + 1).
\end{equation*}
\end{cor:additivecorollary}

\paragraph*{Outline.}
The paper is organised as follows. In \Cref{sec:prelims} we give an introduction into $2$-bridge links, triangulations, and angle structures.
In \Cref{sec:SW} we briefly describe the Sakuma-Weeks triangulations of $2$-bridge link complements, pointing to existing literature where possible.
\Cref{sec:nonmin} contains our new proof of non-minimality of some of the Sakuma-Weeks triangulations. 
\Cref{sec:min} is the heart of the paper. It contains the volume estimates for all $2$-bridge links
covered by \Cref{conj:minimal}. While this involves a lengthy and technical calculuation, it is
not difficult from a conceptual point of view. Finally, in \Cref{ssec:bounds}, we state our complexity bounds resulting from this calculation. 

\paragraph*{Acknowledgements.} The authors thank Stephan Tillmann for the initial ideas that lead to this paper and for useful discussions during its development. The authors also thank the anonymous referee for many helpful comments that improved this paper. We thank Technische Universit\"{a}t Berlin and Freie Universit\"{a}t Berlin for hosting the authors during October and November 2023.
The work of both authors is partially supported by the Australian Research Council under the Discovery Project scheme (grant number DP190102259). James Morgan is additionally supported by an Australian Government Research Training Program (RTP) Scholarship.

\section{Preliminaries}
\label{sec:prelims}

\subsection{$2$-bridge links}
We follow the setup from Appendix A in \cite{gueritaud_canonical_2006}. Let $S$ denote a sphere with four marked points, which we refer to as a \emph{pillowcase}. A four-string braid between two pillowcases, one interior and one exterior, is an embedding of four disjoint arcs into the \emph{product region} $S\times I$, where $I=[a,b]$, such that each arc connects a marked point on the exterior pillowcase $S\times\{b\}$ to a marked point on the interior pillowcase $S\times\{a\}$.
Such an embedding is described by a word
\begin{equation*}
	\Omega = \begin{cases}
		R^{a_1}L^{a_2}\cdots R^{a_n}\quad\text{or}\quad L^{a_1}R^{a_2}\cdots L^{a_n}, & \text{ if }n\text{ is odd},  \\
		R^{a_1}L^{a_2}\cdots L^{a_n}\quad\text{or}\quad L^{a_1}R^{a_2}\cdots R^{a_n}, & \text{ if }n\text{ is even},
	\end{cases}
\end{equation*}
where $a_i\in\Z$. We fix the projection of the braid onto a plane such that each $R$ encodes a twist between the top two strands and each $L$ encodes a twist between the right two strands, shown in \Cref{fig:RandL}. We refer to these as \emph{vertical} and \emph{horizontal} crossings, respectively, and note that the word $\Omega$ encodes a sequence of twists from the outside in; we start at $S\times\{b\}$ and work in towards $S\times\{a\}$.
\begin{figure}[h]
	\centering
	\includegraphics[width=0.5\textwidth]{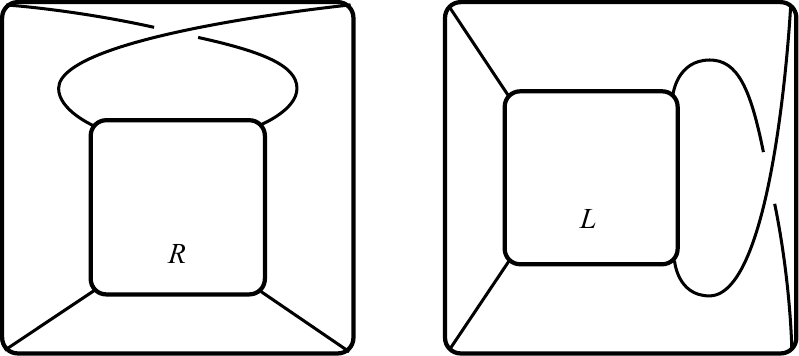}
	\caption{Crossings of the four arcs encoded by $R$ and $L$. A crossing encoded by $R$ is called \emph{vertical} and a crossing encoded by $L$ is called \emph{horizontal}.}
	\label{fig:RandL}
\end{figure}
A \emph{syllable} is a maximal subword in $L$ or $R$. For example, $R^{a}R^{b}$ will occur as the single syllable $R^{a+b}$. The word $\Omega$ consists of $n$ syllables. The sign of each $a_i$ determines whether a crossing is positive or negative and $|a_i|$ determines the number of crossings in the syllable. The crossings associated to a syllable form a \emph{twist region} which is homeomorphic to $S\times I$. A \emph{$2$-bridge link} $K(\Omega)$ is obtained by adding two arcs with a single crossing connecting the opposite pairs of marked points on the exterior pillowcase and two arcs with a single crossing connecting opposite pairs of marked points on the interior pillowcase. This is shown in \Cref{fig:twobridgelink}. Note that every $2$-bridge link can be obtained this way.

\begin{figure}[h]
	\centering
	\includegraphics[scale=0.5]{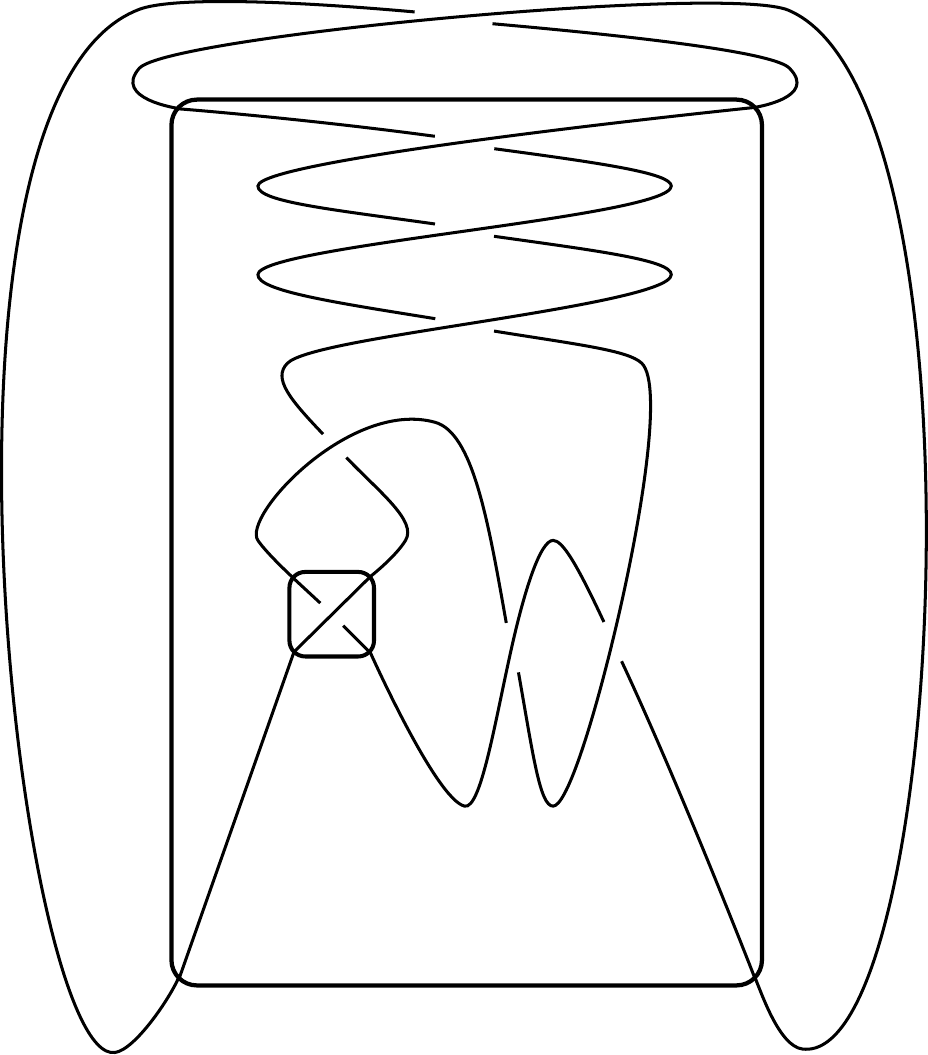}
	\caption{The $2$-bridge link $K(\Omega)$ with $\Omega = R^3L^2R$. Removing the outermost and innermost crossings returns the four-string braid corresponding to $\Omega$.}
	\label{fig:twobridgelink}
\end{figure}

The presentation of a $2$-bridge link by a word $\Omega$ is connected to continued fraction expansions of rational tangles, see \cite{conway_enumeration_1970,purcell_hyperbolic_2020}. Hence we make the following simplifying assumptions:
\begin{itemize}
	\item Given any word $\Omega$ describing a $2$-bridge link $K(\Omega)$ there is another word $\Omega'$ with either all $a_i>0$ or all $a_i<0$ for $1\leq i\leq n$ such that $K(\Omega')$ is ambient isotopic to $K(\Omega)$;
	\item The first syllable of $\Omega$ is $R^{a_1}$; and
	\item $|a_1|\geq 1$ and $|a_n|\geq 1$.
\end{itemize}
Throughout we assume that $a_i>0$ for all $1\leq i\leq n$. The case where all $a_i<0$ is similar.

For the purposes of this paper we concern ourselves only with hyperbolic $2$-bridge links; those whose complements admit a complete hyperbolic metric of finite volume. For each word $\Omega$ the associated $2$-bridge link $K(\Omega)$ is non-split, in the sense that the complement is irreducible, and the chosen projection is alternating. In particular, $K(\Omega)$ is a torus link when $\Omega$ contains exactly one syllable. This imposes the following condition on $\Omega$ due to Menasco.

\begin{theorem}[\cite{menasco_closed_1984}, Corollary 2]
	A $2$-bridge link $K(\Omega)$ is hyperbolic if and only if the associated word $\Omega$ has at least two syllables.
\end{theorem}
\subsection{Triangulations}
A \emph{triangulation} $\tri$ is a decomposition of a $3$-manifold $M$ into tetrahedra complete with instructions on how to glue the tetrahedra together in order to recover $M$. We allow for two faces of the same tetrahedron to be glued together; such a triangulation is sometimes called a \emph{generalised triangulation} in the literature and possesses less restrictions than a simplicial complex. The gluing instructions are done in the standard way, taking vertices to vertices, edges to edges, and faces to faces. 

The $0$-, $1$-, and $2$-dimensional simplices of $\tri$ are called the \emph{vertices}, \emph{edges}, and \emph{faces} of the triangulation, respectively. We call the collection of $k$-dimensional simplices the \emph{$k$-skeleton} of $\tri$ and denote it by $\tri^{(k)}$. Of particular importance to us will be the \emph{edge classes} of $\tri$ contained in $\tri^{(1)}$. An edge class is the equivalence class of edges formed when gluing tetrahedra together. The \emph{degree} of an edge class is the number of edges contained in the equivalence class. We will often refer to these simply as edges when the context is clear.

For this paper we work with (topologically) \emph{ideal triangulations} $\widehat{\tri} = \tri\backslash\tri^{(0)}$. These triangulations are formed from \emph{ideal tetrahedra} - tetrahedra with their vertices removed. The elements of $\tri^{(0)}$ are termed the \emph{ideal vertices} of $\widehat{\tri}$ and the elements of $\widehat{\tri}^{(1)}$ the \emph{ideal edge classes}, or simply the ideal edges of $\widehat{\tri}$.

For more technical information on ideal triangulations we direct the reader to \cite[Section~2]{jaco_minimal_2020}, and \cite[Chapter~4]{purcell_hyperbolic_2020}.

\subsection{Hyperbolic volume and angle structures}
We introduce the basic definitions and results we work with here but omit proofs and instead direct the reader to Futer and Gu\'eritaud \cite{futer_angled_2011}, Milnor \cite{milnor_hyperbolic_1982} or Purcell \cite{purcell_hyperbolic_2020}.

\begin{definition}\label{defn:angle-structure}
	Let $M$ be a $3$-manifold with ideal triangulation $\tri$ consisting of $n$ ideal tetrahedra $\{\Delta_1,\dots,\Delta_n\}$. An \emph{angle structure} $\Theta=(\theta_1,\dots,\theta_{3n})$ on $\tri$ is an assignment of (internal) dihedral angles $\theta_{3i-2},\theta_{3i-1},\theta_{3i}$ on each pair of opposite edges of each ideal tetrahedron $\Delta_i$ in $\tri$ such that the following hold:
	\begin{enumerate}
		\item $\theta_{3i-2},\theta_{3i-1},\theta_{3i}\in(0,\pi)$ for each $1\leq i\leq n$,
		\item For each ideal tetrahedron $\Delta_i$ we have $\theta_{3i-2}+\theta_{3i-1}+\theta_{3i} = \pi$,
		\item The sum of dihedral angles around every (interior) ideal edge class in $\tri$ is $2\pi$.
	\end{enumerate}
	The set of all angle structures on a triangulation is denoted $\ang(\tri)$. The assignment of dihedral angles to the $i$--th ideal tetrahedron is notated as $\Theta^{(i)} = (\theta_{3i-2},\theta_{3i-1},\theta_{3i})$.
\end{definition}
Our definition is sometimes referred to as a \emph{strict} or \emph{positive} angle structure in the literature. There is no reason why $\ang(\tri)$ should be non-empty, however if it is non-empty, then it is the interior of a convex, compact polytope in $[0,\pi]^{3n}\subset\R^{3n}$.

Considering an ideal tetrahedron as a hyperbolic ideal tetrahedron -- the convex hull of four distinct points in $\partial\mathbb{H}^3$ -- the assignment of internal dihedral angles to pairs of opposite edges determines its shape up to isometry \cite{futer_angled_2011}. We denote the hyperbolic ideal tetrahedra with dihedral angles $\{\theta_{3i-2},\theta_{3i-1},\theta_{3i}\}$ as $\Delta(\theta_{3i-2},\theta_{3i-1},\theta_{3i})$.

\begin{theorem}
	Suppose $\theta_{3i-2},\theta_{3i-1},\theta_{3i}\in(0,\pi)$ and $\theta_{3i-2}+\theta_{3i-1}+\theta_{3i}=\pi$. Then these angles determine a hyperbolic ideal tetrahedron $\Delta(\theta_{3i-2},\theta_{3i-1},\theta_{3i})$ with volume
	\begin{equation*}
		\Vol(\Delta(\theta_{3i-2},\theta_{3i-1},\theta_{3i})) = \Lambda(\theta_{3i-2}) + \Lambda(\theta_{3i-1}) + \Lambda(\theta_{3i}),
	\end{equation*}
	where $\Lambda:\R\to\R$ is the Lobachevsky function
	\begin{equation*}
		\Lambda(\theta) = \int_0^\theta \log|2\sin u|du.
	\end{equation*}
\end{theorem}

\begin{definition}\label{defn:volume-functional}
	Let $M$ be a $3$-manifold with ideal triangulation $\tri$ consisting of $n$ ideal tetrahedra. The \emph{volume functional} $\mathcal{V}:\ang(\tri)\to\R$ is defined by
	\begin{equation*}
		\mathcal{V}(\Theta) = \sum_{i=1}^n \Vol(\Delta(\theta_{3i-2},\theta_{3i-1},\theta_{3i})).
	\end{equation*}
\end{definition}

We conclude with the following theorem, proved independently by Casson and Rivin, which follows from proofs in \cite{rivin_euclidean_1994}, and a direct consequence by applying it to Thurston's lower bound on complexity \Cref{eq:volc}.

\begin{theorem}
	Let $M$ be the interior of a compact, orientable $3$-manifold with boundary consisting of tori. Let $\tri$ be an ideal triangulation of $M$ and suppose $\ang(\tri)\neq\emptyset$. A strict angle structure $\Theta\in\ang(\tri)$ corresponds to a complete hyperbolic metric on $M$ if and only if $\Theta$ is the unique critical point of the volume functional $\mathcal{V}:\ang(\tri)\to\mathbb{R}$.
\end{theorem}

\begin{corollary}\label{cor:volumeestimate}
	Let $M$ be a hyperbolic $3$-manifold with ideal triangulation $\tri$. Suppose $\ang(\tri)$ is non-empty. For each $\Theta\in\ang(\tri)$,
	\begin{equation*}
		\mathcal{V}(\Theta) \leq \Vol(M).
	\end{equation*}
\end{corollary}

\section{Sakuma-Weeks triangulations of 2-bridge link complements}
\label{sec:SW}

We now outline the construction of the Sakuma-Weeks triangulations for the complements of $2$-bridge links and detail the combinatorics needed for our results. We follow the setup described in Section 2.1. Further details of this construction can be found in \cite{gueritaud_canonical_2006,purcell_hyperbolic_2020,sakuma_examples_1995}. Throughout this section we define $\Omega = R^{a_1}L^{a_2}\cdots(R^{a_n}\mid L^{a_n})$ with $K(\Omega)$ denoting the associated $2$-bridge link, where $(R^{a_n}\mid L^{a_n})$ denotes the regular expression determining either the string $R^{a_n}$ or $L^{a_n}$ which we assume is done appropriately. We assume that $a_i>0$ for all $1\leq i\leq n$ and that $a_1,a_n\geq 1$. We set $\ell = \sum_{i=1}^n a_i$. Let $M = S^3\backslash K(\Omega)$ denote the link complement. It is vital to first note that we consider a 2-bridge link as being embedded in $S^3$; we suppress this for most of what follows. 

\subsection{Construction of the Sakuma-Weeks triangulations}
Each crossing of the four-string braid determined by $\Omega$ occurs inside a product region homeomorphic to $S\times I$. We denote the complement of these regions as $S_i\times I$ for $1\leq i \leq \ell$, where $S_i$ is a four-punctured sphere and $I=[0,1]$. The Sakuma-Weeks triangulation is constructed by first creating isotopic ideal triangulations of $S_i\times\{1\}$ and $S_i\times\{0\}$ for each $i$. Each such triangulation contains three pairs of edges - vertical, horizontal and diagonal which follow the perspective shown in \Cref{fig:spheretriangulation}.

\begin{figure}[htbp]
	\centering
	\includegraphics[width=\textwidth]{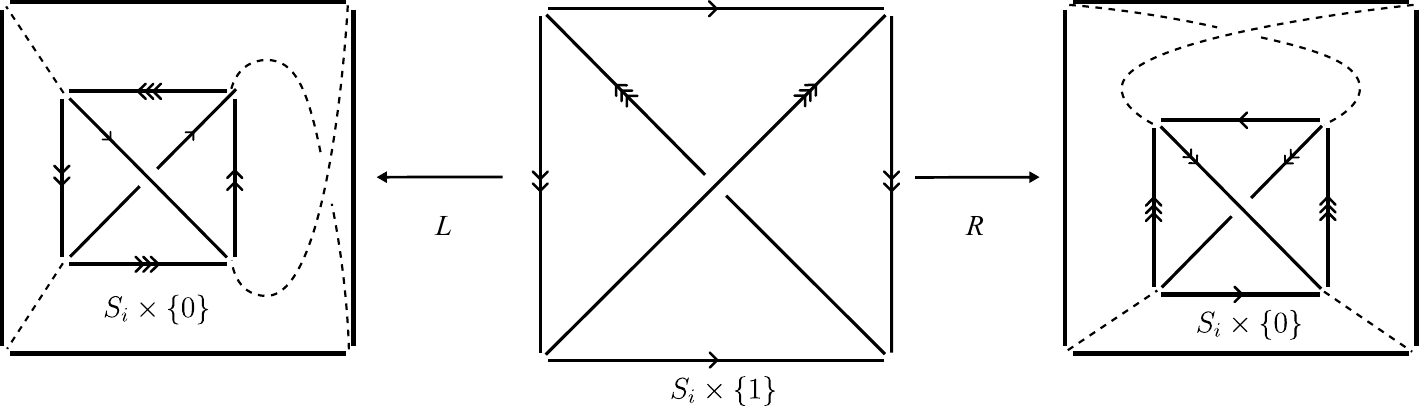}
	\caption{Ideal triangulation of $S_i\times\{1\}$ (outside) and $S_i\times\{0\}$ (inside) with orientations marked on each pair of edges. The complements of the four arcs define an isotopy between the ideal triangulations. Going from $S_i\times\{1\}$ to $S_i\times\{0\}$, a vertical twist $R$ exchanges vertical and diagonal edges and a horizontal twist $L$ exchanges horizontal and diagonal edges, as indicated by the braid arcs (dashed lines).}
	\label{fig:spheretriangulation}
\end{figure}

By gluing $S_i\times\{0\}$ and $S_{i+1}\times\{1\}$ along their horizontal and vertical edges we obtain two regions, each bounded by four ideal vertices, six ideal edges, and four triangular faces. This forms a pair of ideal tetrahedra for each $1\leq i\leq \ell-1$. We call each such pair of tetrahedra a \emph{layer} and denote this by $\widetilde{\Delta}_i$. These layers are glued together by identifying the diagonal edges in $S_{i+1}\times\{1\}$ to either the horizontal or vertical edges in $S_{i+1}\times\{0\}$ for $1\leq i\leq \ell-2$. This process is shown in \Cref{fig:tetrahedralayer}. There remains four unidentified faces in $\widetilde{\Delta}_1$ and four unidentified faces in $\widetilde{\Delta}_{\ell-1}$ coming from $S_{1}\times\{0\}$ and $S_{\ell}\times\{1\}$, respectively. In order to complete the triangulation of the complement of the $2$-bridge link $K(\Omega)$ we identify the four remaining ideal triangles on $S_1\times\{0\}$ in pairs following the isotopy defined by the crossing in $S_1$ and the added exterior crossing. Similarly, for the four remaining faces on $S_{\ell}\times\{1\}$, we follow the isotopy defined by the crossing in $S_{\ell}$ and the added interior crossing. The identification on $S_{\ell-1}\times\{0\}$ is shown in \Cref{fig:closingoff}. Since we assume that the $\Omega$ starts with $R^{a_1}$ we can note that the ideal triangles on $S_1\times\{1\}$ are identified similarly to those on $S_{\ell-1}\times\{0\}$ in the case where $\Omega$ ends with $R^{a_n}$. The resulting triangulation $\tri=\tri(M)$ consists of layers $\widetilde{\Delta}_1,\dots,\widetilde{\Delta}_{\ell-1}$, each consisting of two ideal tetrahedra $\Delta_{2i-1}$ and $\Delta_{2i}$ for $1\leq i\leq \ell-1$. This gives the size of the Sakuma-Weeks triangulation as $|\tri| = 2(\ell-1) = 2\sum_{i=1}^n a_i - 2.$

\begin{figure}[h]
	\centering
	\includegraphics[scale=0.9]{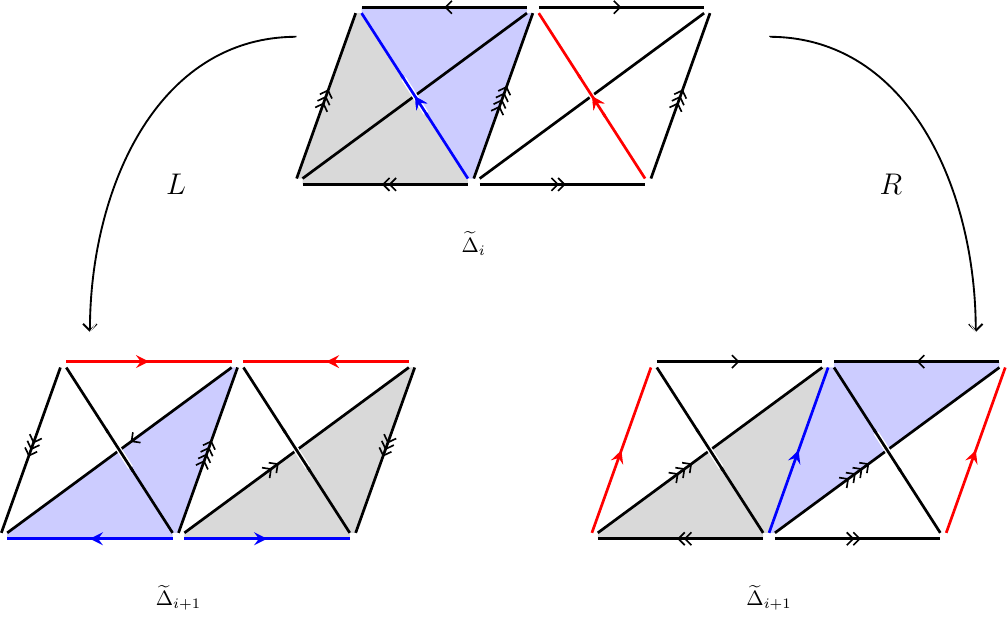}
	\caption{The process of gluing the layers of the Sakuma-Weeks triangulation together. Each layer consists of two ideal tetrahedra glued along their horizontal and vertical edges. The four faces on the back of $\widetilde{\Delta_i}$ lie on $S_i\times\{0\}$ and the four faces on the front lie on $S_{i+1}\times\{1\}$ and similarly for $\widetilde{\Delta_{i+1}}$. We have highlighted two faces on $S_{i+1}\times\{1\}$ and the faces they glue to in $S_{i+1}\times\{0\}$.}
	\label{fig:tetrahedralayer}
\end{figure}

\begin{figure}[h]
	\centering
	\begin{subfigure}[t]{\textwidth}
		\centering
		\includegraphics[scale=0.9]{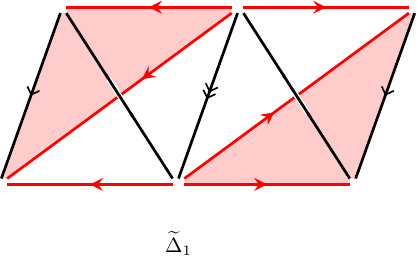}
		\caption{Identifications of ideal triangles on $S_{1}\times\{0\}$. The shaded faces are identified and the unshaded faces are identified.}
		\label{fig:closing-off-a}
	\end{subfigure}
	
	\bigskip
	
	\begin{subfigure}[t]{\textwidth}
		\centering
		\includegraphics[scale=0.9]{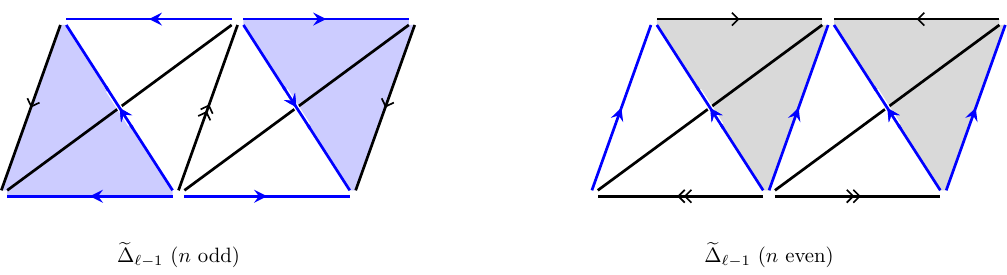}
		\caption{Identifications of ideal triangles on $S_{\ell}\times\{1\}$ when $\Omega$ ends with $R$ (left) and $L$ (right). For each case, the shaded faces are identified and the unshaded faces are identified.}
		\label{fig:closing-off-b}
	\end{subfigure}
	
	\caption{Identifications of remaining ideal triangles in $\widetilde{\Delta}_1$ and $\widetilde{\Delta}_2$. For each layer shown, the back faces belong to $S_i\times\{0\}$ and the front faces belong to $S_{i}\times\{1\}$. Note that no two faces of a single ideal tetrahedron are identified.}
	\label{fig:closingoff}
\end{figure}

As a final note on the construction of the Sakuma-Weeks triangulation, we can observe that the gluing of two layers together forms, in essence, a layered triangulation. In \Cref{fig:tetrahedralayer} we layer the blue and red diagonal edges in $\widetilde{\Delta}_i$ over opposite pairs of horizontal or vertical edges in $\widetilde{\Delta}_{i+1}$ for horizontal ($L$) and vertical ($R$) twists, respectively.

\subsection{Edge Degrees}
The Sakuma-Weeks triangulations of $2$-bridge link complements possess many attractive properties; the construction can be read directly from the link diagram or done algorithmically by reading the word $\Omega$. Our interest lies in computing the degree of each edge class in the triangulation $\tri=\tri(M)$.

The isotopy defined by the crossing between $S_i\times\{1\}$ and $S_i\times\{0\}$ combined with the gluing of $S_i\times\{0\}$ and $S_{i+1}\times\{1\}$ along vertical and horizontal edges allows us to determine each edge class of $\tri$ explicitly and hence compute their degrees. \Cref{fig:tetrahedralayer} shows the tracking of edges through the triangulation with \Cref{fig:closingoff} showing the edge identifications coming from the gluing of faces on $S_1\times\{0\}$ (\Cref{fig:closing-off-a}) and $S_\ell\times\{1\}$ (\Cref{fig:closing-off-b}). We make two immediate observations from these figures. 

First, there are no edges of degrees 1 or 2 in $\tri$. Each layer $\widetilde{\Delta}_i$ consists of two ideal tetrahedra glued along four edges. The corresponding edge classes must have degree of at least 2. Each such edge class also includes at least one diagonal edge. Hence the minimum edge degree is 3. Second, opposite edges in any layer (e.g. horizontal) belong to edge classes of the same degree. This follows since the isotopies are the same on opposite pairs of edges, up to preserving orientation. 

The following lemma provides the number of edges in $\tri$. This result is not new and can be found, for example, in \cite{futer_angled_2011}.

\begin{lemma}\label{lem:numberofedges}
	Let $M$ be the interior of a compact $3$-manifold $\widehat{M}$ with $\partial\widehat{M}\neq\emptyset$ consisting only of a finite number of tori. Let $\tri$ be a topologically ideal triangulation of $M$. Then $|\tri^{(1)}| = |\tri|$, where $\tri^{(1)}$ denotes the one-skeleton of $\tri$. 
\end{lemma}
\begin{proof}
	The triangulation of the boundary tori is obtained by truncating the ideal tetrahedra in $M$, giving $4|\tri|$ triangles along with $12|\tri|$ edges identified in pairs for a total of $6|\tri|$ edges. From the Euler characteristic we have
	\begin{equation*}
		0 = \chi(\partial\widehat{M}) = v - 6|\tri| + 4|\tri|,
	\end{equation*}
	giving $2|\tri|$ vertices. Since each pair of vertices corresponds to a single edge in $\tri$ we obtain the desired equality.
\end{proof}

Minimal triangulations of $3$-manifolds often place restrictions on the possible anatomy of the triangulation. One key restriction is the presence of low-degree edges \cite{jaco_minimal_2020,jaco_minimal_2009,jaco_2thurston_2020,jaco_mathbb_2013}. We conclude this section with necessary and sufficient conditions for the Sakuma-Weeks triangulation to possess edges of degree 3 or 4.

\begin{lemma}\label{lem:degree3}
	Let $K(\Omega)$ be the $2$-bridge link generated by the word $\Omega=R^{a_1}L^{a_2}\cdots(R^{a_n}\mid L^{a_n})$. Let $\tri=\tri(M)$ be the Sakuma-Weeks triangulation of the complement $M=S^3\backslash K(\Omega)$. $\tri$ contains an edge of degree 3 if and only if $a_1>1$ or $a_n>1$.
\end{lemma}
\begin{proof}
	Any edge class containing two diagonal edges must have even degree. This can be determined from \Cref{fig:tetrahedralayer,fig:closingoff}. Hence any edge class with odd degree must contain either the vertical edges from $\Delta_1$ or the horizontal or vertical edges from $\widetilde{\Delta}_{\ell-1}$ if the last syllable of $\Omega$ is $L^{a_n}$ or $R^{a_n}$, respectively.
	
	If $a_1=1$, then the first two letters of $\Omega$ are $RL$ and we see that the vertical edges in $\widetilde{\Delta}_1$ are glued to vertical edges in $\widetilde{\Delta}_2$. If $a_1>1$, however, then the vertical edges in $\widetilde{\Delta}_1$ are glued to distinct diagonal edges in $\widetilde{\Delta}_2$ giving edge classes of degree 3. The argument for $a_n>1$ is analagous, noting that the last letter of $\Omega$ only determines how the remaining four faces are identified on the last layer.
\end{proof}

\begin{lemma}\label{lem:degree4}
	Let $K(\Omega)$ be the $2$-bridge link generated by the word $\Omega=R^{a_1}L^{a_2}\cdots(R^{a_n}\mid L^{a_n})$. Let $\tri=\tri(M)$ be the Sakuma-Weeks triangulation of the complement $M=S^3\backslash K(\Omega)$. $\tri$ contains an edge of degree 4 if and only if $a_i\geq 2$ for some $1< i< n$ or $a_i\geq 3$ for $i\in\{1,n\}$.
\end{lemma}
\begin{proof}
	From \Cref{fig:tetrahedralayer,fig:closingoff} we observe the only way to form an edge class of degree 4 is to glue edges in the sequence diagonal--vertical/horizontal--diagonal. This sequence is only obtainable by performing the same crossing at least twice in succession, with the first diagonal coming from an ideal tetrahedron in the previous layer. Hence we require $a_i\geq 2$ for some $1<i<n$. The first letter of $\Omega$ introduces the first diagonal edges to start this sequence whilst the last letter does not introduce any new ideal tetrahedra. Hence we require that $a_1\geq 3$ or $a_n\geq 3$ for an edge class of degree 4 to contain an edge in $\widetilde{\Delta}_1$ or $\widetilde{\Delta}_{\ell-1}$.
\end{proof}

\section{Non-minimality via elementary moves}
\label{sec:nonmin}

\subsection{A necessary condition for minimality}
In order to show that a triangulation is non-minimal one could simply find a different triangulation possessing fewer tetrahedra. This is often hard. A common approach is to locally change the given triangulation in such a way that we decrease the number of tetrahedra whilst simultaneously preserving the manifold. We are particularly interested in the local modifications called the Pachner 3--2 move and the 4--4 move.

For the Pachner 3--2 move consider a triangulation of a triangular bipyramid consisting of three distinct tetrahedra glued along an internal edge of degree 3. The Pachner 3--2 move replaces this triangulation with one consisting of two distinct tetrahedra glued along an internal triangle (\Cref{subfig:pachner32move}). For the 4--4 move, consider an octahedron triangulated with four distinct tetrahedra glued along an internal edge of degree 4. This edge is realised as a main diagonal of the octahedron which can occur in three distinct ways giving rise to three distinct triangulations of the octahedron. The 4--4 move replaces one of these triangulations with one of the other two \cite{burton_computational_2013}. Crucially, the 4--4 move adjusts the edge degrees of each edge class incident to the tetrahedra by either $\pm1$ or $0$. The edges whose degree decreases are indicated in \Cref{subfig:aggregate44move}.

\begin{figure}[h]
	\centering
	\begin{subfigure}[t]{0.49\textwidth}
		\centering
		\includegraphics[width=0.9\textwidth]{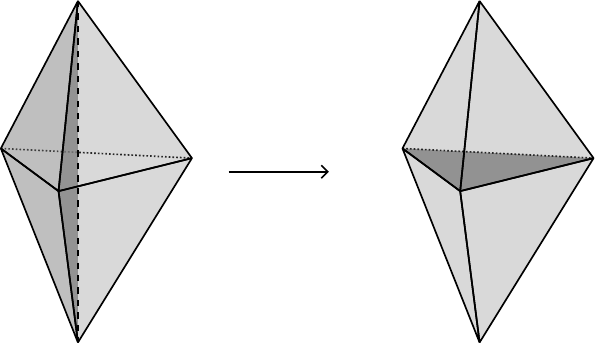}
		\caption{The Pachner 3-2 move.}
		\label{subfig:pachner32move}
	\end{subfigure}
	\begin{subfigure}[t]{0.49\textwidth}
		\centering
		\includegraphics[width=0.9\textwidth]{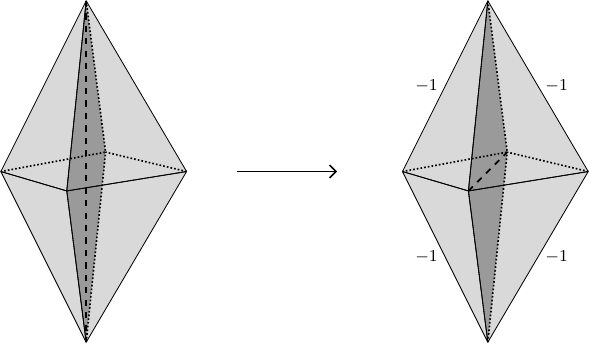}
		\caption{The 4-4 move with negative changes in edge degrees marked.}
		\label{subfig:aggregate44move}
	\end{subfigure}
	\caption{Two elementary moves.}
\end{figure}

\begin{theorem}\label{thm:nonminimal}
	Let $K(\Omega)$ be the $2$-bridge link generated by the word $\Omega = R^{a_1}L^{a_2}\cdots(R^{a_n}\mid L^{a_n})$. Let $\tri = \tri(M)$ be the Sakuma-Weeks triangulation of the complements $S^3\backslash K(\Omega)$. $\tri$ is minimal only if $a_1=a_n = 1$ and $a_i\in\{1,2\}$ for $1<i<n$ with $n\geq 2$.
\end{theorem}
Ishikawa and Nemoto \cite{ishikawa_construction_2016} provide an upper bound on the complexity of the complements of hyperbolic $2$-bridge links as 
\begin{equation*}
	c(M) \leq \sum_{i=1}^n a_i + 2(n-1) - \#\{a_i = 1\},
\end{equation*}
where we have adjusted the formula for our notation of $\Omega$. This result agrees with \Cref{thm:nonminimal} and quantifies a lower bound on the amount of ideal tetrahedra that can be removed.
\begin{proof}
	If $a_1>1$ or $a_n>1$, then there is a degree 3 edge by \Cref{lem:degree3}. By the construction of the Sakuma-Weeks triangulation we know that no two faces of a single ideal tetrahedron are identified and there must be three distinct ideal tetrahedra glued around this edge. Applying a Pachner 3--2 move reduces the size of the triangulation by one and thus $\tri$ is not minimal. 
	
	Suppose now that $a_i\geq 3$ for some $1\leq i\leq n$ and that the corresponding syllable in $\Omega$ is $R^{a_i}$. There are two degree 4 edges, $\edge$ and $\edge'$, incident to each other. \Cref{fig:layersfor44move} shows the four layers of $\tri$ containing $\edge$ (shown in red) and $\edge'$ (shown in blue) with the four ideal tetrahedra glued around $\edge$ shaded. We see that $\edge'$ is contained in this octahedron. Applying a 4--4 move replacing $\edge$ with $\widehat{\edge}$ in the octahedron so that $\widehat{\edge}$ and $\edge'$ are not incident reduces the degree of $\edge'$ by one giving $\operatorname{deg}(\edge') = 3$. Observe that the two faces incident to $\edge'$ now belong to the same ideal tetrahedron and, as we can deduce from \Cref{fig:layersfor44move}, there are two ideal tetrahedra glued around $\edge'$ not contained in the octahedron. Hence there are three distinct tetrahedra glued around $\edge'$ after performing the 4--4 move and we can apply a Pachner 3--2 move to reduce the size of the triangulation by one and thus $\tri$ is not minimal.
	
	Finally, we note that if $a_i\in\{1,2\}$ for all $1\leq i\leq n$ then no two degree 4 edges will be incident to each other in the triangulation and thus we cannot perform a sequence of 4--4 moves to create a degree 3 edge in order to simplify the triangulation.
	\begin{figure}[h]
		\centering
		\includegraphics[scale=0.8]{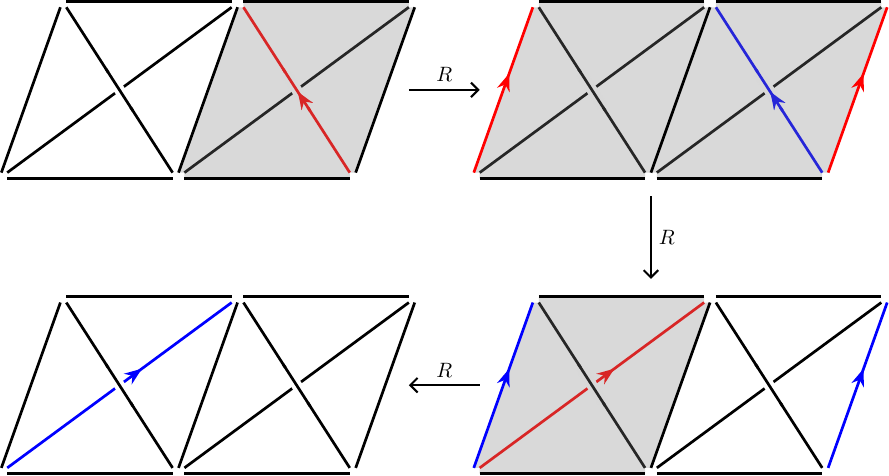}
		\caption{The four layers of $\tri$ containing incident degree 4 edges corresponding to $R^3$ in $\Omega$. The four shaded ideal tetrahedra are glued around the red edge $\edge$ to form an octahedron which contains the blue edge $\edge'$.}
		\label{fig:layersfor44move}
	\end{figure}
\end{proof}

\subsection{Example of a simplification}

We conclude this section with two examples of a $2$-bridge links whose Sakuma-Weeks triangulation is non-minimal, demonstrating the use of 4--4 and Pachner 3--2 moves to simplify this using \emph{Regina} \cite{regina}.

Consider $K(\Omega)$ generated by the word $\Omega=R^2LR$. The Sakuma-Weeks triangulation contains 6 ideal tetrahedra with gluings described by \Cref{tbl:gluing1}. All numbers referenced in this section are given with respect to the numbering provided by Regina given this gluing table.

The triangulation contains two edges of degree 3, listed as Edge 1 and Edge 3. Performing a Pachner 3--2 move along Edge 1 returns a triangulation with 5 ideal tetrahedra. The isomorphism signature of the simplified triangulation is {\tt{fLLQcbcdeeetsfxxh}}.

Now consider $K(\Omega')$ generated by the word $\Omega' = RL^3R$. The Sakuma-Weeks triangulation contains 8 ideal tetrahedra with gluings described by \Cref{tbl:gluing2}. To perform the 4--4 move we use Edge 4 and perform the 4--4 move along Axis 0 to ensure the new edge is not incident to the edge whose degree we want to decrease. The new triangulation obtained has isomorphism signature {\tt{iLLMLQcbcdefhghhmvftgafqa}}. This triangulation contains one degree 3 edge with 3 distinct ideal tetrahedra glued around it. Selecting this and performing a Pachner 3--2 move produces a triangulation with 7 idea tetrahedra and isomorphism signature {\tt{hLLMPkbcdfggfgmvfafwkf}}.

{\renewcommand{\arraystretch}{1.1}
\begin{table}[h]
	\centering\footnotesize
	\parbox{0.48\textwidth}{\begin{tabular}{c|c|c|c|c}
			\hline
			Tetrahedron & Face 012 & Face 013 & Face 023 & Face 123 \\
			\hline
			0           & 3 (102)  & 1 (213)  & 1 (021)  & 2 (023)  \\ \hline
			1           & 0 (032)  & 2 (103)  & 3 (123)  & 0 (103)  \\ \hline
			2           & 4 (032)  & 1 (103)  & 0 (123)  & 5 (321)  \\ \hline
			3           & 0 (102)  & 4 (031)  & 5 (021)  & 1 (023)  \\ \hline
			4           & 5 (032)  & 3 (031)  & 2 (021)  & 5 (103)  \\ \hline
			5           & 3 (032)  & 4 (213)  & 4 (021)  & 2 (321)
		\end{tabular}
		\caption{\small Regina gluing table for the Sakuma-Weeks triangulation of $K(\Omega)$ generated by $\Omega = R^2LR$.}
		\label{tbl:gluing1}}
	\quad
	\parbox{0.48\textwidth}{\begin{tabular}{c|c|c|c|c}
			\hline
			Tetrahedron & Face 012 & Face 013 & Face 023 & Face 123 \\
			\hline
			0           & 2 (032)  & 1 (213)  & 1 (021)  & 3 (321)  \\ \hline
			1           & 0 (032)  & 2 (031)  & 3 (021)  & 0 (103)  \\ \hline
			2           & 4 (032)  & 1 (031)  & 0 (021)  & 5 (321)  \\ \hline
			3           & 1 (032)  & 4 (031)  & 5 (021)  & 0 (321)  \\ \hline
			4           & 6 (032)  & 3 (031)  & 2 (021)  & 7 (321)  \\ \hline
			5           & 3 (032)  & 6 (031)  & 7 (021)  & 2 (321)  \\ \hline
			6           & 7 (032)  & 5 (031)  & 4 (021)  & 7 (103)  \\ \hline
			7           & 5 (032)  & 6 (213)  & 6 (021)  & 4 (321)
		\end{tabular}
		\caption{\small Regina gluing table for the Sakuma-Weeks triangulation of $K(\Omega')$ generated by $\Omega'=RL^3R$.}
		\label{tbl:gluing2}
	}
\end{table}
}

\section{Complexity bounds via angle structures}
\label{sec:min}

In the previous section we provide necessary conditions for the Sakuma-Weeks triangulations of $2$-bridge link complements to be minimal. In this section we provide complexity bounds on the remaining triangulations.

Given a triangulation $\tri$ of a hyperbolic $3$-manifold $M$, Thurston's bound on complexity \cite{thurston_geometry_1978} can be extended to provide both an upper- and lower-bound as
\begin{equation}\label{eq:full-bound}
	\frac{\Vol(M)}{v_3}\leq c(M) \leq |\tri|,
\end{equation}

where $v_3\approx 1.0149\dots$ is the volume of a regular ideal tetrahedron. 

The construction of the Sakuma-Weeks triangulations is symmetric. The layerings used when adding the two ideal tetrahedra from the next layer are the same with the only difference being that we layer each tetrahedron on opposite edges. It is also known that these triangulations have a non-empty space of angle structures. Combining these leads to the following result.

\begin{lemma}[\cite{purcell_hyperbolic_2020}, Lemma 10.24]\label{lem:anglesymmetry}
	Let $\tri$ be a Sakuma-Weeks triangulation of a $2$-bridge link complement with $2n$ tetrahedra. Let $\Delta_{2i-1}$ and $\Delta_{2i}$ be the two tetrahedra in the $i$-th layer $\widetilde{\Delta}_i$ with angles $(\theta_{3i-2}^1,\theta_{3i-1}^1,\theta_{3i}^1)$ and $(\theta_{3i-2}^2,\theta_{3i-1}^2,\theta_{3i}^2)$, respectively. The volume function $\mathcal{V}:\ang(\tri)\to\R$ obtains a maximum when the angles for $\Delta_{2i-i}$ agree with those for $\Delta_{2i}$ for all $1\leq i\leq n$.
\end{lemma}

We utilise this lemma to simplify the construction of explicit angle structures. For the remainder of this paper, in an abuse of notation, we refer to both the $i$-th layer and the tetrahedra it contains as $\Delta_i$. Throughout this section we assume the conditions from \Cref{thm:nonminimal}.
\subsection{Block decomposition of 2-bridge links}
Let $\Omega'$ denote the subword of $\Omega$ obtained by removing the first and last letters. Then for $n\geq 1$,
\begin{equation}\label{eq:subwordOmega}
	\Omega' \vcentcolon= \begin{cases}
		L^{a_1}R^{a_2}\cdots L^{a_n}, & \text{ for odd } n,  \\
		L^{a_1}R^{a_2}\cdots R^{a_n}, & \text{ for even } n.
	\end{cases}
\end{equation}
We define a decomposition of $\Omega'$ via regular expressions to characterise subwords through the use of two operations:
\begin{itemize}
	\item $(\omega)^m$ determines $m$ copies of the string $\omega$,
	\item $(\alpha|\omega)$ determines either the string $\alpha$ or the string $\omega$.
\end{itemize}
These operations allow us, for instance, to write (\ref{eq:subwordOmega}) simply as $\Omega' = L^{a_1}R^{a_2}\cdots(L^{a_n}|R^{a_n})$.
\begin{definition}\label{defn:blocks}
	A \emph{block} is a collection of consecutive syllables matching one of the following types:
	\begin{itemize}
		\item[(B1)] $(L^p(RL)^m\mid R^p(LR)^m)$ for $m\geq 0$ and $p\in\{0,1\}$;
		\item[(B2)] $((L^2)^p(R^2L^2)^m\mid (R^2)^p(L^2R^2)^m)$ for $m\geq 0$ and $p\in\{0,1\}$;
		\item[(B3)] A sequence of $(k+1)$ blocks of type (B1) of length one separated by $k$ blocks of type (B2).
	\end{itemize}
	The \emph{length} of a block is the number of syllables it contains.
\end{definition}
The above blocks can be described simply. (B1) describes strings of alternating $L$'s and $R$'s and (B2) describes strings of alternating $L^2$'s and $R^2$'s. A (B3) block may look like $R(L^2)R(L^2R^2L^2R^2)L(R^2L^2)R$.

The length of a block is not defined in terms of the number of letters the block contains since we aim to decompose words by their syllables. This means that no syllable is split between two different blocks. By our definition we can see that (B1) and (B2) blocks have length $2m+p$ and a (B3) block consisting of $k$ (B2) blocks has length $(k+1) + \sum_{i=1}^k (2m_i+p_i)$ where $2m_i+p_i$ is the length of the $i$-th (B2) block. Given a fixed word $\Omega'$, we call a block of $\Omega'$ {\em maximal} if it cannot be extended to a larger block (possibly of a different type).

\begin{lemma}\label{lem:decomposition}
	Let $\Omega' = L^{a_1}R^{a_2}\cdots(R^{a_n}\mid L^{a_n})$ with $a_i\in\{1,2\}$ for all $1\leq i \leq n$. Then $\Omega'$ can be decomposed into maximal sequences of (B3) blocks, possibly separated by (B1) blocks, and such that (B2) blocks occur at most at the start or end of $\Omega'$. 
\end{lemma}
\begin{proof}
	If $a_1=2$ or $a_n=2$, then the decomposition consists of a (B2) block at the start or end $\Omega'$, respectively. Assume this is not the case and that $a_1=a_n=1$. We show that the decomposition of $\Omega'$ consists only of (B1) and (B3) blocks.
	
	Let $(a_i,\dots,a_{i+k})$ be a maximal sequence of exponents such that $a_j=1$ for $i\leq j\leq i+k$ with $k\geq 0$. This sequence has length $k+1$. If $n=k+1$, then $i=1$ and the sequence corresponds to a (B1) block. Hence let $n>k+1$. If $i=1$, then $a_{i+k+1}=2$ and the (possibly empty) subsequence $(a_1,\dots,a_{i+k-1})$ corresponds to a (B1) block with $(a_{i+k},a_{i+k+1})$ corresponding to the start of a (B3) block.  Similarly, if $i+k=n$, then $a_{i-1}=2$ and $(a_{i-1},a_i)$ corresponds to the end of a (B3) block followed by a (possibly empty) (B1) block. Otherwise $a_{i-1}=a_{i+k+1}=2$. If $k=0$, then the sequence is contained in a (B3) block; if $k\geqq1$, then the sequence corresponds to the end of a (B3) block followed by a (possibly empty) (B1) block followed by the start of another (B3) block.
	
	From this, the maximal sequences of exponents equal to one determine the decomposition of $\Omega'$ with (B2) blocks occurring only at the start or end of $\Omega'$.
\end{proof}

\subsection{Angle structures on blocks}

We show that a decomposition of the Sakuma-Weeks triangulation $\tri$ associated to a word $\Omega$ can be endowed with an angle structure  $\Theta\in\ang(\tri)$ which can be computed directly from the decomposition of $\Omega'$ into blocks of types (B1), (B2) and (B3). Given a block of $\Omega'$ corresponding to one of the types in \Cref{defn:blocks} we can find the corresponding subcomplex $\mathcal{S}\subseteq\tri$. The edge classes which glue to adjacent subcomplexes are incomplete. From condition (2) of \Cref{defn:angle-structure}, the angle sum around any edge class in $\tri$ adds to $2\pi$. We define the \emph{angle deficit} as the remaining angle required to achieve $2\pi$ on such incomplete edge classes. We call an edge an \emph{interior} edge if its edge class is entirely contained in a single block, otherwise we call it a \emph{boundary} edge. Each block has three boundary edge classes at both the start and end of the corresponding subcomplex, each with an angle deficit. We call the triple of angle deficits at the start or end of a block a $\delta$- or $\epsilon$-\emph{boundary deficit}, respectively. Recall from \Cref{defn:angle-structure} that $\Theta = (\theta_1,\dots,\theta_{3\ell}) \in \ang(\tri)$ where $\ell = |\Omega|-1$ and the angles assigned to the $i$-th tetrahedron are $\Theta^{(i)}=(\theta_{3i-2},\theta_{3i-1},\theta_{3i})$. We define the boundary deficits on a block $\delta = (\delta_0,\delta_1,\delta_2)$ and $\epsilon = (\epsilon_0,\epsilon_1,\epsilon_2)$ in the same way (\Cref{fig:boundarypattern}).

\begin{figure}[ht]
	\centering
	\includegraphics[width=0.6\textwidth]{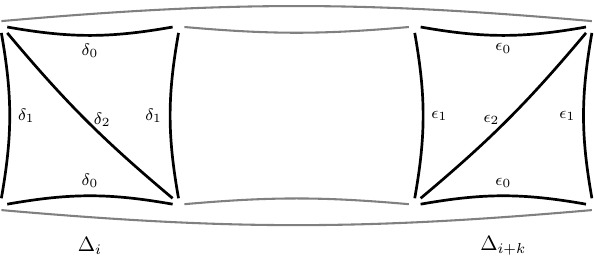}
	\caption{Assignment of $\delta = (\delta_0,\delta_1,\delta_2)$ and $\epsilon=(\epsilon_0,\epsilon_1,\epsilon_2)$ boundary deficits on a block containing $k+1$ layers, starting from layer $\Delta_i$ and containing all layers up to (and including) $\Delta_{i+k}$.}
	\label{fig:boundarypattern}
\end{figure}

Ishikawa and Nemoto \cite{ishikawa_construction_2016} prove that the Sakuma-Weeks triangulation of the $2$-bridge link complement $S^3\backslash L(\Omega)$ is minimal when the word $\Omega$ can be decomposed as a single (B1) block. Their proof shows that the volume of the link complement in this case is almost $|\tri|\cdot v_3$, with a deficit of approximately $0.66$ in total. We generalise their work by assigning a dihedral angle of $\pi/3$ to each edge in a (B1) block ensuring these tetrahedra are regular ideal.

\begin{lemma}\label{lem:angleB1}
	Let $\tri$ be the Sakuma-Weeks triangulation of a $2$-bridge link complement. A subcomplex $\mathcal{S}\subset\tri$ corresponding to a (B1) block can be constructed from only regular ideal tetrahedra. The first and last letter of the block determine the $\delta-$ and $\epsilon-$boundary deficits, respectively as
	\begin{itemize}
		\item $L$ at start: $\delta = (\pi/3,\;\pi,\;5\pi/3)$; $L$ at end: $\epsilon = (\pi,\;\pi/3,\;5\pi/3)$
		\item $R$ at start: $\delta = (\pi,\;\pi/3,\;5\pi/3)$; $R$ at end: $\epsilon = (\pi/3,\;\pi,\;5\pi/3)$
	\end{itemize}
\end{lemma}
\begin{proof}
	Let $\Omega' = L^p(RL)^m$ be the (B1) block. When $m=0, p=1$ we have the following angle equations associated with the boundary edges of $\mathcal{S}$,
	\begin{align*}
		\delta_0 + 2\theta_{3i-2} + \epsilon_0 & = 2\pi, \\
		\delta_1 + 2\theta_{3i-1} + \epsilon_1 & = 2\pi, \\
		\delta_2 + \theta_{3i}                 & = 2\pi, \\
		\theta_{3i} + \epsilon_2               & = 2\pi,
	\end{align*}
	where $\delta_i$ and $\epsilon_i$ denote the boundary deficits at the start and end of the block, respectively. Setting $\theta_j=\pi/3$ for each $3i-2\leq j\leq 3i$ gives,
	\begin{align*}
		\delta_0 + \epsilon_0 = \delta_1 + \epsilon_1 = \frac{4\pi}{3} \qquad\text{and}\qquad \delta_2 = \epsilon_2 = \frac{5\pi}{3}.
	\end{align*}
	This satisfies the claimed boundary deficits.
	When $m=1,p=0$ we obtain the following angle equations associated with the boundary edges of $\mathcal{S}$,
	\begin{align*}
		\delta_0 + 2\theta_{3i-2} + \theta_{3(i+1)}                 & = 2\pi, \\
		\delta_1 + 2\theta_{3i-1} + 2\theta_{3(i+1)-1} + \epsilon_1 & = 2\pi, \\
		\delta_2 + \theta_{3i}                                      & = 2\pi, \\
		\theta_{3i} + 2\theta_{3(i+1)-2} + \epsilon_0               & = 2\pi, \\
		\theta_{3(i+1)} + \epsilon_2                                & =2\pi.
	\end{align*}
	Setting $\theta_j=\pi/3$ for each $3i-2\leq j\leq 3(i+1)$ gives
	\begin{align*}
		\delta_0 = \epsilon_0 = \pi,\qquad \delta_1 + \epsilon_1 = \frac{2\pi}{3} \qquad\text{and}\qquad \delta_2 = \epsilon_2 = \frac{5\pi}{3}.
	\end{align*}
	Setting $\delta_1 = \epsilon_1 = \pi/3$ gives the desired boundary deficit.
	
	Suppose that for $\Omega'$ we achieve the desired boundary deficits with each tetrahedron being regular ideal for some $m\geq 1$ and $p\in\{0,1\}$. The angle equation an interior edge of the block first appearing in $\Delta_{i+k}$, $k\geq 0$, is either
	\begin{align*}
		\theta_{3(i+k)} + 2\theta_{3(i+k+1)-2} + 2\theta_{3(i+k+2)-2} + \theta_{3(i+k+3)} & =2\pi,\qquad\text{ or } \\
		\theta_{3(i+k)} + 2\theta_{3(i+k+1)-1} + 2\theta_{3(i+k+2)-1} + \theta_{3(i+k+3)} & =2\pi,
	\end{align*}
	In each case, setting each angle to be $\pi/3$ satisfies these equations. The angle equations for boundary edge classes at the start of the block are,
	\begin{equation*}
		\begin{aligned}
			\delta_0 + 2\theta_{3i-2} + 2\theta_{3(i+1)-2} + \theta_{3(i+2)} & = 2\pi, \\
			\delta_1 + 2\theta_{3i-1} + \theta_{3(i+1)}                      & =2\pi,  \\
			\delta_2 + \theta_{3i}                                           & =2\pi,
		\end{aligned}\qquad\text{ or }\qquad\begin{aligned}
			\delta_0 + 2\theta_{3i-2} + \theta_{3(i+1)}                      & = 2\pi, \\
			\delta_1 + 2\theta_{3i-1} + 2\theta_{3(i+1)-1} + \theta_{3(i+2)} & = 2\pi, \\
			\delta_2 + \theta_{3i}                                           & = 2\pi.
		\end{aligned}
	\end{equation*}
	for $p=1$ and $p=0$, respectively. Setting $\theta_j=\pi/3$ for each $3i-2\leq j\leq 3(i+1)$ gives the desired boundary deficit. The angle equations for the boundary edge classes at the end of the block are then,
	\begin{align*}
		\theta_{3(i+m-1)} + 2\theta_{3(i+m)-2} + \epsilon_0                        & = 2\pi \\
		\theta_{3(i+m-2)} + 2\theta_{3(i+m-1)-1} + 2\theta_{3(i+m)-1} + \epsilon_1 & = 2\pi \\
		\theta_{3(i+m)} + \epsilon_2                                               & =2\pi.
	\end{align*}
	Setting $\theta_j=\pi/3$ for each $3(i+m-2)\leq j\leq 3(i+m)$ gives the desired boundary deficit with $\epsilon_0 = \pi$, $\epsilon_1=\pi/3$ and $\epsilon_2 = 5\pi/3$.
	
	The proof follows similarly if we assume instead that $\Omega' = R^p(LR)^m$.
\end{proof}

It remains to show that we can assign angles to the remaining block types which are compatible with the boundary deficits on the (B1) blocks. As there are many ways to achieve this, we endeavour to assign angles which achieve a larger volume for the block. We restrict to the following shapes of ideal tetrahedra,
\begin{equation*}
	\arraycolsep=15pt\def\arraystretch{1.7}
	\begin{array}{lll}
		\text{(0) } \Delta(\pi/3,\pi/3,\pi/3)          & \text{(V) } \Delta(\pi/6,\; \pi/2,\; \pi/3)      \\
		\text{(I) } \Delta(\pi/3,\;3\pi/8,\;7\pi/24)   & \text{(VI) } \Delta(\pi/6,\; \pi/4,\; 7\pi/12)   \\
		\text{(II) } \Delta(\pi/3,\;\pi/4,\;5\pi/12)   & \text{(VII) } \Delta(\pi/8,3\pi/8,\pi/2)         \\
		\text{(III) } \Delta(\pi/4,\; \pi/4,\; \pi/2)  & \text{(VIII) } \Delta(\pi/8,\; \pi/4,\; 5\pi/8)  \\
		\text{(IV) } \Delta(5\pi/24,\;7\pi/24,\;\pi/2) & \text{(IX) } \Delta(\pi/12,\; 7\pi/12,\; \pi/3).
	\end{array}
\end{equation*}
\begin{remark}
	The above shapes can be refined further to obtain slightly better volume estimates for our main result. However, such a refinement introduces more shapes and more complicated arguments decreasing the overall readability of the remainder of this section. At the same time the increase in volume estimates is not sufficient to sharpen our main result significantly enough to justify the extra burden put on the reader.
\end{remark}

\begin{lemma}\label{lem:angleB3}
	Let $\tri$ be the Sakuma-Weeks triangulation of a $2$-bridge link complement. A subcomplex $\mathcal{S}\subset\tri$ corresponding to a (B3) block can be constructed using only ideal tetrahedra of shape types (I), (II), (III), (IV), (VI) and (VIII) assigned depending on the number of (B2) blocks as:
	\begin{itemize}
		\item a single (B2) block of length $k$
		      \begin{center}
			      \begin{tikzpicture}
				      \node (I) {(I)};
				      \node[right=of I] (VI) {(VI)};
				      \node[right=of VI] (III) {\big[\,(III)\,\big]$\,\big.^{2(k-1)}$};
				      \node[right=of III] (IV) {(IV)};
				      \node[right=of IV] (II) {(II)};
				      
				      \draw[->] (I) -- (VI);
				      \draw[->] (VI) -- (III);
				      \draw[->] (III) -- (IV);
				      \draw[->] (IV) -- (II);
			      \end{tikzpicture}
		      \end{center}
		\item $k$ (B2) blocks of length 1,
		      \begin{center}
			      \begin{tikzpicture}
				      \node (I) {(I)};
				      \node[right=of I] (VI) {(VI)};
				      \node[right=of VI] (IIIi) {\Big[\,(III)};
				      \node[right=of IIIi] (III) {(III)};
				      \node[right=of III] (VIII) {(VIII)\,\Big]$\,\Big.^{k-1}$};
				      \node[right=of VIII] (IV) {(IV)};
				      \node[right=of IV] (II) {(II)};
				      
				      \draw[->] (I) -- (VI);
				      \draw[->] (VI) -- (IIIi);
				      \draw[->] (IIIi) -- (III);
				      \draw[->] (III) -- (VIII);
				      \draw[->] (VIII) -- (IV);
				      \draw[->] (IV) -- (II);
			      \end{tikzpicture}
		      \end{center}
		\item $k$ (B2) blocks of arbitrary lengths combines the above two by adding more layers of shape (III).
	\end{itemize}
	
	The first and last letter of the block determines the $\delta-$ and $\epsilon-$boundary deficits, respectively as
	\begin{itemize}
		\item $L$ at start: $\delta = (\pi/3,\;\pi,\;5\pi/3)$; $L$ at end: $\epsilon = (\pi,\;\pi/3,\;5\pi/3)$
		\item $R$ at start: $\delta = (\pi,\;\pi/3,\;5\pi/3)$; $R$ at end: $\epsilon = (\pi/3,\;\pi,\;5\pi/3)$
	\end{itemize}
\end{lemma}
\begin{proof}
	Recall that a (B3) block consists of at least two (B1) blocks of length one separated by nonempty (B2) blocks. The shapes which appear are determined by both the number of (B1) blocks, and the lengths of the (B2) blocks. We assume throughout that our block starts with $L$, however all of the following arguments still hold if it instead starts with $R$ by switching $\theta_{3j-2}$ and $\theta_{3j-1}$ in each $\Delta_j$.
	
	\textbf{One (B2) block -- }
	Let $\Omega' = L(R^2L^2)^m(R^2L)^pR^{1-p}$ with $m\geq 0$, $p\in\{0,1\}$. When $m=0$, $p=1$ we have $\Omega'=LR^2L$ giving the following angle equations,
	\begin{align}\label{eq:121base}
		\begin{split}
			\delta_0 + 2\theta_{3i-2} + 2\theta_{3(i+1)-2} + 2\theta_{3(i+2)-2} + \theta_{3(i+3)} &= 2\pi,\\
			\delta_1 + 2\theta_{3i-1} + \theta_{3(i+1)} &= 2\pi,\\
			\delta_2 + \theta_{3i} &= 2\pi,\\
			\theta_{3i} + 2\theta_{3(i+1)-1} + \theta_{3(i+2)} &= 2\pi,\\
			\theta_{3(i+1)} + 2\theta_{3(i+2)-1} + 2\theta_{3(i+3)-1} + \epsilon_1 &=2\pi,\\
			\theta_{3(i+2)} + 2\theta_{3(i+3)-2} + \epsilon_0 &= 2\pi,\\
			\theta_{3(i+3)} + \epsilon_2 &= 2\pi,
		\end{split}
	\end{align}
	where $\delta_j$ and $\epsilon_j$ denote the boundary deficits at the start and end of the block, respectively. We obtain tetrahedra of types (I), (II), (IV) and (VI) by setting,
	\begin{equation}\label{eq:121open-and-close}
		\begin{alignedat}{9}
			\text{(I)} & \quad & \angT{i} &= \angles{\frac{7\pi}{24}}{\frac{3\pi}{8}}{\frac{\pi}{3}}, & \qquad & 
			\text{(VI)} & \quad & \angT{i+1} &= \angles{\frac{\pi}{6}}{\frac{7\pi}{12}}{\frac{\pi}{4}}, \\
			\text{(IV)} & \quad & \angT{i+2} &= \angles{\frac{5\pi}{24}}{\frac{7\pi}{24}}{\frac{\pi}{2}}, & \qquad & 
			\text{(II)} & \quad & \angT{i+3} &= \angles{\frac{\pi}{4}}{\frac{5\pi}{12}}{\frac{\pi}{3}}.
		\end{alignedat}
	\end{equation}
	Substituting these values into \Cref{eq:121base} results in the following boundary deficits,
	\begin{equation*}
		\delta_0 = \epsilon_1 = \frac{\pi}{3}, \qquad
		\delta_1 = \epsilon_0 = \pi\qquad\text{and}\qquad
		\delta_2 = \epsilon_2 = \frac{5\pi}{3}.
	\end{equation*}
	Proceeding to the case when $m=1$, $p=0$ so that $\Omega' = LR^2L^2R$, we can note that the first four angle equations of \Cref{eq:121base} remain the same and are solved by the same values as above, also preserving the $\delta$-boundary deficit. The remaining angle equations are as follows,
	\begin{align*}
		\begin{split}
			\theta_{3(i+1)} + 2\theta_{3(i+2)-1} + 2\theta_{3(i+3)-1} + 2\theta_{3(i+4)-1} + \theta_{3(i+5)} &=2\pi,\\
			\theta_{3(i+2)} + 2\theta_{3(i+3)-2} + \theta_{3(i+4)} &= 2\pi,\\
			\theta_{3(i+3)} + 2\theta_{3(i+4)-2} + 2\theta_{3(i+5)-2} + \epsilon_0 &= 2\pi,\\
			\theta_{3(i+4)} + 2\theta_{3(i+5)-1} + \epsilon_1 &= 2\pi,\\
			\theta_{3(i+5)} + \epsilon_2 &= 2\pi.
		\end{split}
	\end{align*}
	We set $\angT{i}$ and $\angT{i+1}$ as in \Cref{eq:121open-and-close}. To obtain tetrahedra of shapes (IV) and (II) we set,
	\begin{equation}\label{eq:121closeR}
		\begin{alignedat}{9}
			\text{(IV)} & \quad & \angT{i+4} &=\angles{\frac{7\pi}{24}}{\frac{5\pi}{24}}{\frac{\pi}{2}}, &\qquad& 
			\text{(II)} & \quad & \angT{i+5} &= \angles{\frac{5\pi}{12}}{\frac{\pi}{4}}{\frac{\pi}{3}}.
		\end{alignedat}
	\end{equation}
	This assignment of angles yields $\epsilon_1=\pi$ and $\epsilon_2=5\pi/3$. To satisfy the remaining equations, we choose $\Delta_{i+2}$ and $\Delta_{i+3}$ to be have shape (III) as,
	\begin{equation}\label{eq:121shape3}
		\angT{i+2} = \angles{\frac{\pi}{4}}{\frac{\pi}{4}}{\frac{\pi}{2}} \qquad\text{and}\qquad \angT{i+3} =\angles{\frac{\pi}{2}}{\frac{\pi}{4}}{\frac{\pi}{4}},
	\end{equation}
	yielding $\epsilon_0 = \pi/3$. Observe that in the preceding two cases the first two angle vectors are the same and the last two angle vectors are the same with the first two angles swapped. The order in which these angles are assigned can be determined by the last letter of the block. If $\Omega'$ ends on $L$, then we assign these angles as in \Cref{eq:121open-and-close}, otherwise we set them as in \Cref{eq:121closeR}. For this assignment of angles we always obtain the desired boundary deficits. Similarly, when increasing $m$, if the added pair of syllables is $L^2R^2$ we repeat the assignment of angles in \Cref{eq:121shape3} with the first two angles swapped in each vector.
	
	Suppose now that $m\geq 2$ and $p\in\{0,1\}$. We only need to verify the claim that all tetrahedra not in the first or last two layers are of type (III). The angle equations for the interior edge classes alternate between,
	\begin{align}\label{eq:b3interior1}
		\begin{split}
			\theta_{3(i+k)} + 2\theta_{3(i+k+1)-1} + 2\theta_{3(i+k+2)-1} + 2\theta_{3(i+k+3)-1} + \theta_{3(i+k+4)} &=2\pi,\\
			\theta_{3(i+k+1)} + 2\theta_{3(i+k+2)-2} + \theta_{3(i+k+3)} &= 2\pi.
		\end{split}
	\end{align}
	and
	\begin{align}\label{eq:b3interior2}
		\begin{split}
			\theta_{3(i+k)} + 2\theta_{3(i+k+1)-2} + 2\theta_{3(i+k+2)-2} + 2\theta_{3(i+k+3)-2} + \theta_{3(i+k+4)} &=2\pi,\\
			\theta_{3(i+k+1)} + 2\theta_{3(i+k+2)-1} + \theta_{3(i+k+3)} &= 2\pi.
		\end{split}
	\end{align}
	for $k\geq 1$. In the case that \Cref{eq:b3interior1} is followed by \Cref{eq:b3interior2}, obtained by replacing $(i+k)$ in the latter equation with $(i+k+2)$, we can solve the equations by setting the relevant tetrahedra to have shape (III) as
	\begin{equation*}
		\begin{aligned}
			\angT{i+k+1} & = \angles{\frac{\pi}{4}}{\frac{\pi}{4}}{\frac{\pi}{2}} \\
			\angT{i+k+3} & = \angles{\frac{\pi}{4}}{\frac{\pi}{4}}{\frac{\pi}{2}} \\
		\end{aligned} \qquad \begin{aligned}
			\angT{i+k+2} & = \angles{\frac{\pi}{2}}{\frac{\pi}{4}}{\frac{\pi}{4}} \\
			\angT{i+k+4} & = \angles{\frac{\pi}{4}}{\frac{\pi}{2}}{\frac{\pi}{4}},
		\end{aligned}
	\end{equation*}
	where this sequence of angle vectors is repeated as necessary. If instead \Cref{eq:b3interior2} is followed by \Cref{eq:b3interior1} we simply swap $\angT{i+k+2}$ and $\angT{i+k+4}$ above. 
	
	\textbf{Multiple (B2) blocks of length one -- }
	Let $\Omega'=L(R^2L)^m$. When $m=2$ we have $\Omega'=LR^2LR^2L$ giving the following angle equations,
	\begin{align}\label{eq:121case2}
		\begin{split}
			\delta_0 + 2\theta_{3i-2} + 2\theta_{3(i+1)-2} + 2\theta_{3(i+2)-2} + \theta_{3(i+3)} &= 2\pi,\\
			\delta_1 + 2\theta_{3i-1} + \theta_{3(i+1)} &= 2\pi,\\
			\delta_2 + \theta_{3i} &= 2\pi,\\
			\theta_{3i} + 2\theta_{3(i+1)-1} + \theta_{3(i+2)} &= 2\pi,\\
			\theta_{3(i+1)} + 2\theta_{3(i+2)-1} + 2\theta_{3(i+3)-1} + \theta_{3(i+4)} &= 2\pi,\\
			\theta_{3(i+2)} + 2\theta_{3(i+3)-2} + 2\theta_{3(i+4)-2} + 2\theta_{3(i+5)-2} + \theta_{3(i+6)} &= 2\pi,\\
			\theta_{3(i+3)} + 2\theta_{3(i+4)-1} + \theta_{3(i+5)} &= 2\pi,\\
			\theta_{3(i+4)} + 2\theta_{3(i+5)-1} + 2\theta_{3(i+6)-1} + \epsilon_1 &= 2\pi,\\
			\theta_{3(i+5)} + 2\theta_{3(i+6)-2} + \epsilon_0 &= 2\pi,\\
			\theta_{3(i+6)} + \epsilon_2 &=2\pi.
		\end{split}
	\end{align}
	
	Noting that $\Omega'$ ends on $L$, we set $\Theta^{(i)}$, $\Theta^{(i+1)}$, $\Theta^{(i+5)}$ and $\Theta^{(i+6)}$ as in \Cref{eq:121open-and-close} in order to obtain tetrahedra of types (I), (VI), (IV) and (II), respectively. To obtain tetrahedra of type (III) we set,
	\begin{equation}
		\angT{i+2} = \angles{\frac{\pi}{4}}{\frac{\pi}{4}}{\frac{\pi}{2}} \qquad \text{and} \qquad
		\angT{i+3} = \angles{\frac{\pi}{4}}{\frac{\pi}{2}}{\frac{\pi}{4}}.
	\end{equation}
	Note that the part of $\Omega'$ these angles are associated to is the first $R^2L$ and hence the $\pi/2$ appears as the second angle in $\Theta^{(i+3)}$, rather than as the first angle. Solving the remaining equations gives us a layer of shape (VIII) as
	\begin{equation}\label{eq:121shapeVIII}
		\angT{i+4} = \angles{\frac{\pi}{8}}{\frac{5\pi}{8}}{\frac{\pi}{4}}.
	\end{equation}
	Substituting these angles into \Cref{eq:121case2} results in the following boundary deficits,
	\begin{equation*}
		\delta_0 = \epsilon_1 =\frac{\pi}{3}, \qquad
		\delta_1 = \epsilon_0 =\pi\qquad\text{and}\qquad
		\delta_2 = \epsilon_2 = \frac{5\pi}{3}.
	\end{equation*}
	Suppose that for some $m>2$ we can assign angles as above with tetrahedra shapes (I) and (VI) appearing at the start of the block, shapes (IV) and (II) appearing at the end of the block and all other tetrahedra having shapes (III) and (VIII) following the pattern of $\Theta^{(i+2)},\Theta^{(i+3)}$ and $\Theta^{(i+4)}$ above. Consider $\Omega' = L(R^2L)^{m+1}$. The modified and new angle equations are
	\begin{align}\label{eq:modandnew}
		\begin{split}
			\theta_{3(i+3m-1)} + 2\theta_{3(i+3m)-1} + 2\theta_{3(i+3m+1)-1} + \theta_{3(i+3m+2)} &= 2\pi,\\
			\theta_{3(i+3m)} + 2\theta_{3(i+3m+1)-2} + 2\theta_{3(i+3m+2)-2} + 2\theta_{3(i+3m+3)-2} + \theta_{3(i+3m+4)} &= 2\pi,\\
			\theta_{3(i+3m+1)} + 2\theta_{3(i+3m+2)-1} + \theta_{3(i+3m+3)} &= 2\pi,\\
			\theta_{3(i+3m+2)} + 2\theta_{3(i+3m+3)-1} + 2\theta_{3(i+3m+4)-1} + \epsilon_1 &= 2\pi,\\
			\theta_{3(i+3m+3)} + 2\theta_{3(i+3m+4)-2} + \epsilon_0 &=2\pi,\\
			\theta_{3(i+3m+4)} + \epsilon_2 &= 2\pi.
		\end{split}
	\end{align}
	By our assumption $\Delta_{i+3m-1}$ has shape (VIII). We set $\Delta_{i+3m+3}$ and $\Delta_{i+3m+4}$ to have shapes (IV) and (II), respectively, as in \Cref{eq:121open-and-close}. Substituting these values into the above equations we find
	\begin{equation*}
		\begin{alignedat}{9}
			\text{(III)} & \quad & \angT{i+3m} &= \angles{\frac{\pi}{4}}{\frac{\pi}{4}}{\frac{\pi}{2}}, &\qquad&
			\text{(III)} & \quad & \angT{i+3m+1} &= \angles{\frac{\pi}{4}}{\frac{\pi}{2}}{\frac{\pi}{4}}, \\
			\text{(VIII)} & \quad & \angT{i+3m+2} &= \angles{\frac{\pi}{8}}{\frac{5\pi}{8}}{\frac{\pi}{4}}. & {} & {} & {} & {} & {}
		\end{alignedat}
	\end{equation*}
	This assignment of angles gives the desired $\epsilon$-boundary deficit of $\epsilon_0=\pi$, $\epsilon_1=\pi/3$ and $\epsilon_2 = 5\pi/3.$
	
	\textbf{At least two (B2) blocks with arbitrary length --}
	The final case combines the previous two cases. Starting from the second case, we claim that increasing the length of any (B2) block adds additional tetrahedra of shape (III) whilst preserving the shapes of the remaining tetrahedra. Note that it follows from the preceding arguments that if there are at least two (B2) blocks, then increasing the length of the first or last (B2) block does not affect the equations containing the first or last boundary deficit, respectively, up to permuting the first two angles.
	
	Suppose $\Omega' = L(R^2L)^m$ with $m\geq 2$ with shapes described above. Consider the subword $\Omega''=LR^2$ consisting of the first two syllables of $\Omega'$ and the relevant subcomplex $\mathcal{S}''\subset\mathcal{S}$. Then $\mathcal{S}''$ has boundary deficits along the edges which glue to $\mathcal{S}\backslash\mathcal{S}''$. We show that these deficits are preserved up to permutation of the first two angles if we extend the (B2) block in $\Omega''$.
	
	By assumption $\Delta_i$, $\Delta_{i+1}$ and $\Delta_{i+2}$ have shapes (I), (VI) and (III), respectively, assigned as in the previous case. Substituting into lines 1, 5 and 6 of \Cref{eq:121case2}, which are the angle equations containing the deficits in $\mathcal{S}''$, we have,
	\begin{align*}
		\delta_0 + \frac{14\pi}{24} + \frac{2\pi}{6} + \frac{\pi}{2} + \mu_0 & = 2\pi, \\
		\frac{\pi}{4} + \frac{2\pi}{4} + \mu_1                               & = 2\pi, \\
		\frac{\pi}{2} + \mu_2                                                & = 2\pi,
	\end{align*}
	where $\mu_0,\mu_1,\mu_2$ denote the angle deficits on $\mathcal{S}''$, replacing the angles from $\angT{k}$ for $k\geq i+3$ in \Cref{eq:121case2}. Since $\delta_0 = \pi/3$ we solve the above equations as
	\begin{equation*}
		\mu_0 = \frac{\pi}{4},\qquad \mu_1 = \frac{5\pi}{4} \qquad\text{and}\qquad \mu_2 = \frac{3\pi}{2}.
	\end{equation*}
	Suppose we now extend the length of the (B2) block in $\Omega''$ by one. This corresponds to a subword of $\Omega' = (LR^2L^2)(RL^2)^q R$ with $q = m-1$. By assumption we have $\angT{i}$ and $\angT{i+1}$ as above and further that $\angT{i+2},\angT{i+3}$ and $\angT{i+4}$ have shape (III) assigned as
	\begin{equation*}
		\angT{i+2} = \left(\frac{\pi}{4}, \frac{\pi}{4}, \frac{\pi}{2}\right),\quad \angT{i+3} = \left(\frac{\pi}{2},\frac{\pi}{4},\frac{\pi}{4}\right),\quad\text{and}\quad \angT{i+4} = \left(\frac{\pi}{4},\frac{\pi}{4},\frac{\pi}{2}\right).
	\end{equation*}
	The relevant equations containing the new deficits of $\mathcal{S}''$ are 
	\begin{align*}
		\theta_{3(i+1)} + 2\theta_{3(i+2)-1} + 2\theta_{3(i+3)-1} + 2\theta_{3(i+4)-1} + \eta_0 &= 2\pi, \\
		\theta_{3(i+3)} + 2\theta_{3(i+4)-2} + \eta_1 &= 2\pi, \\
		\theta_{3(i+4)} + \eta_2 &=2\pi,
	\end{align*}
	where $\eta_0,\eta_1,\eta_2$ denote the new deficits on $\mathcal{S}''$. Substituting the angles in we have
	\begin{align*}
		\frac{\pi}{4} + \frac{2\pi}{4} + \frac{2\pi}{4} + \frac{2\pi}{4}+ \eta_0 & = 2\pi, \\
		\frac{\pi}{4} + \frac{2\pi}{4} + \eta_1                                  & = 2\pi, \\
		\frac{\pi}{2} + \eta_2                                                   & = 2\pi,
	\end{align*}
	Solving the above equations we see that
	\begin{equation*}
		\eta_0 = \frac{\pi}{4},\qquad \eta_1 = \frac{5\pi}{4} \qquad\text{and}\qquad \eta_2 = \frac{3\pi}{2}.
	\end{equation*}
	Since $\mu_i=\eta_i$ for each $i=0,1,2$, extending any (B2) block in $\Omega'$ must introduce ideal tetrahedra with shape (III). The remaining shapes of the tetrahedra in both instances remain the same. However, after we extend the (B2) block we must swap the first two angles in each remaining $\angT{j}$ as described in Cases 1 and 2 above. We also observe that the angle deficits seen here are the same if we follow this procedure at the end of any (B2) block in $\Omega'$.
\end{proof}

Given the angle structures on (B1) and (B3) blocks we now check that these angle structures are valid when we glue a (B1) block to a (B3) block.

\begin{lemma}\label{lem:B1andB3}
	The angle structures on (B1) and (B3) blocks are compatible. More precisely, a (B1) block can always be glued to either the start or the end of a (B3) block.
\end{lemma}
\begin{proof}
	Denote the first and last letters of a (B1) and (B3) block by $\alpha_{(B1)}$, $\alpha_{(B3)}$, and $\omega_{(B1)}$, $\omega_{(B3)}$, respectively.
	Let $\epsilon_0$, $\epsilon_1$ and $\epsilon_2$ denote the boundary deficits at the end of the (B1) block as in \Cref{fig:boundarypattern}. Similarly, let $\delta_0$, $\delta_1$ and $\delta_2$ be the boundary deficits at the start of the (B3) block. One key take away from \Cref{lem:angleB1,lem:angleB3} is that the $\epsilon$-boundary deficit of a (B1) block equals the $\delta$-boundary deficit of a block of (B3) as long as $\omega_{(B1)} \neq \alpha_{(B3)}$. In particular, we have $\delta_i = \epsilon_i$, $0 \leq i \leq 2$ whenever $\omega_{(B1)} \neq \alpha_{(B3)}$. The same is true for the end of a (B3) block and the start of a (B1) block.
	
	For the angle structures from the blocks to extend to a valid angle structure after gluing, the boundary angles of the identified edges must sum to $2\pi$. If $\omega_{(B1)} = L$ and $\alpha_{(B3)}=R$ we have identifications leading to $\epsilon_0 + \delta_0 = \epsilon_2 + \delta_1 = \epsilon_1 + \delta_2 = 2 \pi$. To see this note that, since $\alpha_{(B3)}=R$, this layering swaps vertical and diagonal edges (indices $1$ and $2$). Since, following \Cref{lem:angleB1,lem:angleB3} we have $\epsilon_0 = \delta_0 = \pi$, $\epsilon_1 = \delta_1 = \pi/3$, and $\epsilon_2 = \delta_2 = 5\pi/3$, these identities hold. If instead we have  $\omega_{(B1)} = R$ and $\alpha_{(B3)} = L$, we have $\epsilon_0 + \delta_2 = \epsilon_1 + \delta_1 = \epsilon_2 + \delta_0 = 2 \pi$. Again, following \Cref{lem:angleB1,lem:angleB3}, we have $\epsilon_0 = \delta_0 = \pi/3$, $\epsilon_1 = \delta_1 = \pi$, $\epsilon_2 = \delta_2 = 5\pi/3$, and the equations are satisfied.
	
	An analogous argument shows that the angle structures are compatible when gluing the end of a (B3) block to the start of a (B1) block.
\end{proof}

\begin{lemma}\label{lem:B1B3Closure}
	Let $\Omega = RL^{a_1}\cdots(R^{a_n}L\mid L^{a_n}R)$ with $n\geq1$ be a word such that $\Omega' = L^{a_1}\cdots(R^{a_n}\mid L^{a_n})$ can be decomposed into only (B1) and (B3) blocks. Let $\tri$ be the Sakuma-Weeks triangulation of the associated $2$-bridge link complement consisting of $\Sigma=1+\sum a_i$ layers. There is an angle structure on $\tri$ such that each block has the angles assigned as in \Cref{lem:angleB1} and \Cref{lem:angleB3} and $\Delta_1$ has shape (V). If the last block in $\Omega'$ is of type (B1), then $\Delta_\Sigma$ has shape (V). If the last block of $\Omega'$ is of type (B3), then $\Delta_{\Sigma-1}$ has shape (III) and $\Delta_\Sigma$ has shape (VIII). In both cases, $\Delta_\Sigma$ replaces the last tetrahedron in the final block.
\end{lemma}
\begin{proof}
	From \Cref{lem:angleB1,lem:angleB3} we have the $\delta$-boundary deficit at the start of the first block
	\begin{equation*}
		\delta_0 = \frac{\pi}{3},\qquad \delta_1 = \pi \quad\text{and}\quad \delta_2 = \frac{5\pi}{3}.
	\end{equation*}
	From \Cref{fig:closing-off-a} we equate this boundary deficit to edges in $\Delta_1$ as,
	\begin{equation*}
		\delta_0 = \theta_{3(1)}, \qquad
		\delta_1 = 2\theta_{3(1)-1}\qquad\text{and}\qquad
		\delta_2 = 4\theta_{3(1)-2} + 2\theta_{3(1)} + \theta_{3(2)}.
	\end{equation*}
	As $\Delta_2$ is the first layer in a (B1) or (B3) block, we have $\theta_{3(2)} = \pi/3$ (see \Cref{eq:121open-and-close} for the case of (B3)). Assigning shape (V) to $\Delta_1$ as
	\begin{equation*}
		\Theta^{(1)} = \left(\theta_1,\theta_2,\theta_3\right) = \left(\frac{\pi}{6},\frac{\pi}{2},\frac{\pi}{3}\right)
	\end{equation*}
	solves the boundary angle equations. Note that extending the length of the block in the case of (B1) replaces the $\delta$-boundary deficit with multiples of $\pi/3$ leaving the above assignment of $\Theta^{(1)}$ to satisfy the equations. 
	
	We now consider the shape of $\Delta_\Sigma$. Suppose that $\Omega=R\cdots R^{a_{n-1}}L^{a_n}R$ so that $n$ is odd. From \Cref{fig:tetrahedralayer,fig:closing-off-b} (using the case of $n$ odd),  the angle equations for edges in $\Delta_\Sigma$ are
	\begin{align}\label{eq:lastlayereqns}
		\begin{split}
			4\theta_{3\Sigma-2} + 2\theta_{3\Sigma} + 2\theta_{3(\Sigma-1)} &= 2\pi,\\
			x + 2\theta_{3\Sigma-1} &= 2\pi,
		\end{split}
	\end{align}
	where $x$ denotes the angle deficit on the vertical edge class of $\Delta_\Sigma$. To verify this, we note that the four horizontal edges and two diagonal edges in $\Delta_{\Sigma}$ are identified and, given $\Omega$, must also include both diagonal edges from the previous layer. This gives the first equation. To obtain the second equation, note that we have two distinct edge classes for the four vertical edges in $\Delta_\Sigma$ which also include some edges from previous layers. We consider two cases for the values of $a_{n-1}$ and $a_n$.
	
	\textbf{Case 1: }
	Suppose that $a_{n} = a_{n-1} = 1$ so that $\Omega'$ ends on a (B1) block. Note that $R^{a_{n-1}}$ either corresponds to a part of the same (B1) block or to the last syllable of a (B3) block. From \Cref{lem:angleB1,lem:angleB3} we know that we have $\theta_{3(\Sigma-1)} = \pi/3$ in both of these cases. Using this, the first line of \Cref{eq:lastlayereqns} above is satisfied by $\theta_{3\Sigma-2} = \pi/6$ and $\theta_{3\Sigma} = \pi/3$.
	
	We now verify that $\theta_{3\Sigma-1}=\pi/2$. This angle corresponds to the vertical edge class of $\Delta_\Sigma$. Remove $L^{a_n}$ from $\Omega'$ and consider it as a (B1) block of length one. Then $R^{a_{n-1}}$ is the end of either a (B1) or (B3) block. We know that the vertical edge class of $\Delta_\Sigma$ glues to the vertical edge class of $\Delta_{\Sigma-1}$ and hence it suffices to check that the angle deficit in $\Delta_{\Sigma-1}$ is $\pi$. It follows from \Cref{lem:angleB1,lem:angleB3} that the angle deficit of the vertical edge class for a (B1) or (B3) block ending on $R$ is indeed $\pi$, and hence $\angT{\Sigma}$ has shape (V).
	
	\textbf{Case 2: }
	Suppose now that $\Omega$ has the same form as above with $a_{n-1} = 2$ and $a_n=1$ so that $\Omega'$ ends on a (B3) block. The angles in the first equation from \Cref{eq:lastlayereqns} remain the same, however we can determine $x$ explicitly from \Cref{fig:tetrahedralayer} as
	\begin{equation*}
		x = \theta_{3(\Sigma-2)} + 2\theta_{3(\Sigma-1)-1}.
	\end{equation*}
	It must be that $\Delta_{\Sigma-2}$ has shape (VI) or (VIII). From \Cref{eq:121open-and-close,eq:121shapeVIII} we must have that $\theta_{3(\Sigma-2)} = \pi/4$. Substituting in the second line of \Cref{eq:lastlayereqns} yields
	\begin{equation*}
		2\theta_{3(\Sigma-1)-1} + 2\theta_{3\Sigma-1} = \frac{7\pi}{4}.
	\end{equation*}
	This is solved by setting $\theta_{3(\Sigma-1)-1} = \pi/4$ and $\theta_{3\Sigma-1} = 5\pi/8$. Solving the remaining equation gives $\theta_{3\Sigma-2} = \pi/8$, $\theta_{3\Sigma} = \pi/4$ and $\theta_{3(\Sigma-1)} = \pi/2$. This gives shapes (III) and (VIII) as
	\begin{equation*}
		\angT{\Sigma-1} = \angles{\frac{\pi}{4}}{\frac{\pi}{4}}{\frac{\pi}{2}} \quad\text{ and }\quad \angT{\Sigma} = \angles{\frac{\pi}{8}}{\frac{5\pi}{8}}{\frac{\pi}{4}}.
	\end{equation*}
	To check compatability we consider the one remaining angle equation which interacts with these shapes,
	\begin{equation*}
		\theta_{3(\Sigma-3)} + 2\theta_{3(\Sigma-2)-1} + \theta_{3(\Sigma-1)} = 2\pi.
	\end{equation*}
	For these shapes to be compatible we require $\theta_{3(\Sigma-3)} + 2\theta_{3(\Sigma-2)-1} = 7\pi/4$. If $\Delta_{\Sigma-2}$ has shape (VI), then \Cref{eq:121open-and-close} gives us $\theta_{3(\Sigma-2)-1} = 7\pi/12$ and $\theta_{3(\Sigma-3)} = \pi/3$. Otherwise both $\Delta_{\Sigma-2}$ and $\Delta_{\Sigma-3}$ have shape (III) and we have $\theta_{3(\Sigma-2)-1} = \pi/2$ and $\theta_{3(\Sigma-3)} = \pi/2$ and conclude that the shapes are compatible.
\end{proof}
We now have valid angle structures for the Sakuma-Weeks triangulation of any $2$-bridge link described by a word $\Omega = RLR^{a_2}\cdots(LR\mid RL)$ such that $a_2,\dots,a_{n-1}\in\{1,2\}$ -- that is, the subword obtained by removing the first and last letter can be decomposed into only (B1) and (B3) blocks.

We now examine the angle structures on (B2) blocks. Recall from \Cref{lem:decomposition} that a (B2) block must occur as the first or last block in the word. We find an angle structure on these blocks which is compatible with (B1) and (B3) blocks on one boundary and satisfies the face identifications corresponding to the first or last letter of $\Omega$ on the other boundary. For reference, we record the $\delta-$ and $\epsilon-$boundary deficits of (B1) and (B3) blocks here.
\begin{equation*}
	\arraycolsep=7pt\def\arraystretch{1.6}
	\begin{array}{|ccc|c|ccc|}
		\hline
		\delta_0 & \delta_1 & \delta_2 &   & \epsilon_0 & \epsilon_1 & \epsilon_2 \\ \hline
		\pi/3    & \pi      & 5\pi/3   & L & \pi        & \pi/3      & 5\pi/3     \\
		\pi      & \pi/3    & 5\pi/3   & R & \pi/3      & \pi        & 5\pi/3     \\ \hline
	\end{array}
\end{equation*}

\begin{lemma}\label{lem:angleB2start}
	Let $\tri$ be the Sakuma-Weeks triangulation of a $2$-bridge link complement associated to the word $\Omega=RL^{a_1}\cdots(L^{a_n}R\mid R^{a_n}L)$ with $a_i=2$ for $1\leq i\leq k$. The subcomplex $\mathcal{S}\subset\tri$ corresponding to the (B2) block at the start of $\Omega'$ can be constructed using only ideal tetrahedra of shape types (I), (III) and (VI) arranged as
	\begin{center}
		\begin{tikzpicture}
			\node (III) {\big[(III)\big]$\big.^{2(k-1)}$};
			\node[right=of III] (VI) {(VI)};
			\node[right=of VI] (I) {(I)};
			
			\draw[->] (III) -- (VI);
			\draw[->] (VI) -- (I);
		\end{tikzpicture}
	\end{center}
	The block is compatible with $\Delta_1$ having shape (VII) and the end of the block is compatible with (B1) and (B3) blocks.
\end{lemma}
\begin{proof}
	Let $\Omega' = (L^2R^2)^{m-1}L^2$ for $m\geq 1$ be the (B2) block. Denote the layers of this block as $\Delta_2,\dots,\Delta_{4m-1}$. The angle equations for the edges of $\Delta_{4m-2}$ and $\Delta_{4m-1}$ are,
	\begin{align}
		\begin{split}
			x_0 + 2\theta_{3(4m-2)-2} + \theta_{3(4m-1)} &= 2\pi,\\
			x_1 + 2\theta_{3(4m-2)-1} + 2\theta_{3(4m-1)-1} + \epsilon_1 &= 2\pi,\\
			x_2 + \theta_{3(4m-2)} &= 2\pi,\\
			\theta_{3(4m-2)} + 2\theta_{3(4m-1)-2} + \epsilon_0 &= 2\pi,\\
			\theta_{3(4m-1)} + \epsilon_2 &= 2\pi,
		\end{split}
	\end{align}
	where $x_0$, $x_1$ and $x_2$ denote angle deficits on the respective edge classes. For any angle structure on this block to be compatible with a (B1) or (B3) block we require that the $\epsilon$-boundary deficit on $\Delta_{4m-1}$ agrees with the $\delta$-boundary deficit of the (B1) or (B3) blocks - which must start on $R$. That is $\epsilon_i = \delta_i$ for $0\leq i \leq 2$. Hence $\epsilon_0 = \pi$, $\epsilon_1 = \pi/3$ and $\epsilon_2=5\pi/3$.
	
	Setting 
	\begin{equation*}
		\text{(VI)} \quad \angT{4m-2}=\angles{\frac{7\pi}{12}}{\frac{\pi}{6}}{\frac{\pi}{4}} \qquad\text{and}\qquad
		\text{(I)} \quad \angT{4m-1}= \angles{\frac{3\pi}{8}}{\frac{7\pi}{24}}{\frac{\pi}{3}}.
	\end{equation*}
	solves the last two equations and determines the angle deficits as
	\begin{equation}\label{eq:B2startdeficits}
		x_0 = \frac{\pi}{2},\qquad x_1=\frac{3\pi}{4}\quad\text{and}\quad x_2 =\frac{7\pi}{4}
	\end{equation}
	Suppose that $m=1$ so that the (B2) block is of length one. The angle deficits above must be satisfied by the angle equations from $\Delta_1$. From \Cref{fig:tetrahedralayer,fig:closing-off-a} we have,
	\begin{align}\label{eq:B2xdeficit}
		\begin{split}
			x_0 &= \theta_{3(1)},\\
			x_1 &= 2\theta_{3(1)-1},\\
			x_2 &= 4\theta_{3(1)-2} + 2\theta_{3(1)} + \theta_{3(2)}.
		\end{split}
	\end{align}
	Setting $\angT{1} = (\pi/8, 3\pi/8, \pi/2)$ gives $\Delta_1$ shape (VII) and satisfies the deficits in \Cref{eq:B2startdeficits}.
	
	Suppose now that $m>1$. We claim that the additional layers between $\Delta_1$ and $\Delta_{4m-2}$ all have shape (III). The angle equations are the same as those seen in the proof of \Cref{lem:angleB3} and so it suffices to check that this assignment is compatible with \Cref{eq:B2startdeficits} and assigning shape (VII) to $\Delta_1$.
	
	We first expand the angle deficits above as,
	\begin{align*}
		x_0 & = \theta_{3(4m-3)},                              \\
		x_1 & = \theta_{3(4m-4)} + 2\theta_{3(4m-3)-1},        \\
		x_2 & = y + 2\theta_{3(4m-4)-2} + 2\theta_{3(4m-3)-2}.
	\end{align*}
	We have $\pi/2 = x_0 = \theta_{3(4m-3)}$ and assume $\Delta_{4m-3}$ has shape (III), hence we assign $\angT{4m-3} = (\pi/4,\pi/4,\pi/2)$. By construction, $y$ must correspond to the horizontal edge class of the previous layer. Observe that $\Omega'$ repeats strings of $(L^2R^2)$. This corresponds to \Cref{eq:b3interior2} followed by \Cref{eq:b3interior1} and hence, from the reasoning following these equations, we have $\angT{4m-4} = (\pi/4,\pi/2,\pi/4)$. This forces $y = 3\pi/4$. From the proof of \Cref{lem:angleB3} (the case of one (B2) block) we know that $y=2\pi/4 + \pi/4$ and thus conclude that $\Delta_{4m-4}$ and $\Delta_{4m-3}$ have shape (III).
	
	To check that this assignment is compatible with $\Delta_1$ having shape (VII), we consider the following angle equations from $\Delta_2$ and $\Delta_3$, assuming $m>1$,
	\begin{align*}
		x_0 + 2\theta_{3(2)-2} + \theta_{3(3)}        & = 2\pi, \\
		x_1 + 2\theta_{3(2)-1} + 2\theta_{3(3)-1} + z & = 2\pi, \\
		x_2 + \theta_{3(2)}                           & = 2\pi,
	\end{align*}
	where the deficits $x_i$ are as in \Cref{eq:B2xdeficit}. Assigning $\angT{2} = (\pi/2,\pi/4,\pi/4)$ and $\angT{3} = (\pi/4,\pi/4,\pi/2)$ gives $\Delta_1$ shape (VII) as above if $z = \pi/4$. Note that $\Delta_4$ must have either shape (III) or shape (VI) and $z$ must correspond to the diagonal edge class. In either case, we have $\theta_{3(4)} = \pi/4$ and thus the shape assignment is compatible.
	
	An analogous argument applies if $\Omega' = (L^2R^2)^m$ for $m\geq 1$. Here, the first two angles in each layer are swapped.    
\end{proof}

\begin{lemma}\label{lem:angleB2end}
	Let $\tri$ be the Sakuma-Weeks triangulation of a $2$-bridge link complement associated to the word $\Omega=RL^{a_1}\cdots(L^{a_n}R\mid R^{a_n}L)$ with $a_i=2$ for $n-k< i\leq n$, $k\geq 2$, and let $\Sigma=1+\sum a_i$. The subcomplex $\mathcal{S}\subset\tri$ corresponding to the (B2) block at the end of $\Omega$ can be constructed using only ideal tetrahedra of shape types (III), (V) and (IX) arranged as
	\begin{center}
		\begin{tikzpicture}
			\node (X) {(IX)};
			\node[right=of X] (V) {\big[(V)\big]$\,^{3}$};
			
			\draw[->] (X) -- (V);
		\end{tikzpicture}
	\end{center}
	for $k=2$ and,
	\begin{center}
		\begin{tikzpicture}
			\node (X) {(IX)};
			\node[right=of X] (III) {\big[(III)\big]$\,^{2k-5}$};
			\node[right=of III] (V) {\big[(V)\big]$\,^4$};
			
			\draw[->] (X) -- (III);
			\draw[->] (III) -- (V);
		\end{tikzpicture}
	\end{center}
	for $k\geq 3$.
	The start of the block is compatible with blocks of types (B1) and (B3) and the end of the block is compatible with the face identifications on $\Delta_\Sigma$.
\end{lemma}

\begin{proof}
	Assume the last letter of $\Omega$ is $R$. Suppose that $k=2$ and the (B2) block is given by $R^2L^2$. The block contains the layers $\Delta_{\Sigma-3}$ to $\Delta_\Sigma$, inclusive. The angle equations of this block are now,
	\begin{align*}
		\delta_0 + 2\theta_{3(\Sigma-3)-2} + 2\theta_{3(\Sigma-2)-2} + \theta_{3(\Sigma-1)}            & = 2\pi, \\
		\delta_1 + 2\theta_{3(\Sigma-3)-1} + \theta_{3(\Sigma-2)}                                      & = 2\pi, \\
		\delta_2 + \theta_{3(\Sigma-3)}                                                                & = 2\pi, \\
		\theta_{3(\Sigma-3)} + 2\theta_{3(\Sigma-2)-1} + 2\theta_{3(\Sigma-1)-1} + 2\theta_{3\Sigma-1} & = 2\pi, \\
		\theta_{3(\Sigma-2)} + 2\theta_{3(\Sigma-1)-2} + \theta_{3\Sigma}                              & = 2\pi, \\
		2\theta_{3(\Sigma-1)} + 4\theta_{3\Sigma-2} + 2\theta_{3\Sigma}                                & = 2\pi.
	\end{align*}
	The block starts on $R$ hence the $\delta$-boundary deficit is given by $\delta_0 = \pi$, $\delta_1 = \pi/3$ and $\delta_2 = 5\pi/3$. We assign shape (IX) to $\Delta_{\Sigma-3}$ as
	\begin{equation*}
		\angT{\Sigma-3} = \angles{\frac{\pi}{12}}{\frac{7\pi}{12}}{\frac{\pi}{3}}.    
	\end{equation*}
	We then assign shape (V) to the remaining shapes as,
	\begin{equation*}
		\angT{\Sigma-2} = \angles{\frac{\pi}{3}}{\frac{\pi}{6}}{\frac{\pi}{2}},\qquad \angT{\Sigma-1} = \angles{\frac{\pi}{2}}{\frac{\pi}{3}}{\frac{\pi}{6}}\quad\text{and}\quad\angT{\Sigma} = \angles{\frac{\pi}{6}}{\frac{\pi}{3}}{\frac{\pi}{2}}.    
	\end{equation*}
	Increasing to $k=3$ the (B2) block is now given by $L^2R^2L^2$. We can observe that the last three angle equations above remain the same. The three boundary angle equations are the same as above, however the angles assigned to vertical and horizontal edges are swapped. From this we assign shape (IX) to $\Delta_{\Sigma-5}$ as was done for $\Delta_{\Sigma-3}$ above, but swapping the first two angles and shape (V) to layers $\Delta_{\Sigma-2}$, $\Delta_{\Sigma-1}$ and $\Delta_{\Sigma}$ exactly as above. The angle equations involving layers $\Delta_{\Sigma-4}$ and $\Delta_{\Sigma-3}$ are,
	\begin{align*}
		\delta_0 + 2\theta_{3(\Sigma-5)-2} + \theta_{3(\Sigma-4)}                                                                 & = 2\pi, \\
		\delta_1 + 2\theta_{3(\Sigma-5)-1} + 2\theta_{3(\Sigma-4)-1} + \theta_{3(\Sigma-3)}                                       & = 2\pi, \\
		\theta_{3(\Sigma-5)} + 2\theta_{3(\Sigma-4)-2} + 2\theta_{3(\Sigma-3)-2} + 2\theta_{3(\Sigma-2)-2} + \theta_{3(\Sigma-1)} & = 2\pi, \\
		\theta_{3(\Sigma-4)} + 2\theta_{3(\Sigma-3)-1} + \theta_{3(\Sigma-2)}                                                     & = 2\pi, \\
		\theta_{3(\Sigma-3)} + 2\theta_{3(\Sigma-2)-1} + 2\theta_{3(\Sigma-1)-1} + 2\theta_{3\Sigma-1}                            & = 2\pi.
	\end{align*}
	Substituting in $\delta_0 = \pi/3$, $\delta_1=\pi$ and the angles from the shape assignments to $\Delta_{\Sigma-5},\Delta_{\Sigma-2},\Delta_{\Sigma-1}$ and $\Delta_{\Sigma}$ allows us to assign shape (III) to $\Delta_{\Sigma-4}$ and (V) to $\Delta_{\Sigma-3}$ as,
	\[
		\angT{\Sigma-4} = \angles{\frac{\pi}{4}}{\frac{\pi}{4}}{\frac{\pi}{2}}\qquad\text{and}\qquad\angT{\Sigma-3}=\angles{\frac{\pi}{6}}{\frac{\pi}{2}}{\frac{\pi}{3}}.
	\]
	Extending the block to $k>3$ adds angle equations as in \Cref{eq:b3interior1,eq:b3interior2} and can be solved by adding additional layers of shape (III).
	
	An analogous argument shows that this assignment of shapes is valid if the (B2) block ends on $R^2$ (giving the last letter of $\Omega$ to be $L$) or starts on $L^2$ with the first two angles of each shape vector swapped.
\end{proof}

We now address the case of $a_{n-1}=1$ and $a_{n}=2$, which is neglected in the above lemma. Extending this block to contain the final letter of $\Omega$ we obtain a (B3) block in which the final (B2) block contained within it has length one. In this situation it is more convenient to treat the subcomplex corresponding to this block as an `unfinished' (B3) block. 
\begin{lemma}\label{lem:unfinishedB3}
	Let $\tri$ be the Sakuma-Weeks triangulation of a $2$-bridge link complement associated to the word $\Omega=RL^{a_1}\cdots(L^{a_n}R\mid R^{a_n}L)$ with $a_{n-1}=1$, $a_{n}=2$, $n\geq 2$, and $\Sigma = 1+\sum a_i$. The subcomplex $\mathcal{S}\subset\tri$ corresponding to an unfinished (B3) block at the end of $\Omega$ containing the layers $\Delta_{\Sigma-2}$, $\Delta_{\Sigma-1}$ and $\Delta_{\Sigma}$ can be constructed as in \Cref{lem:angleB3} with the two layers containing ideal tetrahedra of shapes (IV) and (II) replaced with a single layer containing ideal tetrahedra of shape (VII) or (III) if the block contains one or more (B2) blocks, respectively.
\end{lemma}
\begin{proof}
	Suppose that $\Omega'=LR^2(L)$ so that $\mathcal{S}$ contains only three layers, with $(L)$ denoting the final letter of $\Omega$. The corresponding angle equations are,
	\begin{align*}
		\delta_0 + 2\theta_{3(\Sigma-2)-2} + 2\theta_{3(\Sigma-1)-2} + 2\theta_{3\Sigma-2} & = 2\pi, \\
		\delta_1 + 2\theta_{3(\Sigma-2)-1} + \theta_{3(\Sigma-1)}                          & = 2\pi, \\
		\delta_2 + \theta_{3(\Sigma-2)}                                                    & = 2\pi, \\
		\theta_{3(\Sigma-2)} + 2\theta_{3(\Sigma-1)-1} + \theta_{3\Sigma}                  & = 2\pi, \\
		2\theta_{3(\Sigma-1)} + 4\theta_{3\Sigma-1} + 2\theta_{3\Sigma}                    & = 2\pi.
	\end{align*}
	The block starts on $L$ hence the $\delta$-boundary deficit is given by $\delta_0 = \pi/3$, $\delta_1 = \pi$ and $\delta_2 = 5\pi/3$. We obtain shapes (I), (VI) and (VII) by assigning the angles,
	\begin{equation*}
		\angT{\Sigma-2} = \angles{\frac{7\pi}{24}}{\frac{3\pi}{8}}{\frac{\pi}{3}},\quad \angT{\Sigma-1} = \angles{\frac{\pi}{6}}{\frac{7\pi}{12}}{\frac{\pi}{4}}\quad\text{and}\quad\angT{\Sigma} = \angles{\frac{3\pi}{8}}{\frac{\pi}{8}}{\frac{\pi}{2}}.
	\end{equation*}
	If instead the (B3) block is longer, then we only need to consider an additional two layers preceding $\Delta_{\Sigma-2}$. Suppose that the tail of $\Omega'$ has the form $R^2LR^2(L)$, where $(L)$ denotes the final letter of $\Omega$. The corresponding angle equations are,
	\begin{align*}
		\theta_{3(\Sigma-4)} + 2\theta_{3(\Sigma-3)-1} + 2\theta_{3(\Sigma-2)-1} + \theta_{3(\Sigma-1)} & = 2\pi, \\
		\theta_{3(\Sigma-3)} + 2\theta_{3(\Sigma-2)-2} + 2\theta_{3(\Sigma-1)-2} + 2\theta_{3\Sigma-2}  & = 2\pi, \\
		\theta_{3(\Sigma-2)} + 2\theta_{3(\Sigma-1)-1} + \theta_{3\Sigma}                               & = 2\pi, \\
		2\theta_{3(\Sigma-1)} + 4\theta_{3\Sigma-1} + 2\theta_{3\Sigma}                                 & = 2\pi.
	\end{align*}
	Following \Cref{lem:angleB3} we assign shape (VIII) as 
	\begin{equation*}
		\angT{\Sigma-1} = \angles{\frac{\pi}{8}}{\frac{5\pi}{8}}{\frac{\pi}{4}},
	\end{equation*}
	and shape (III) as
	\begin{equation*}
		\angT{\Sigma-2} = \angles{\frac{\pi}{4}}{\frac{\pi}{4}}{\frac{\pi}{2}}\quad\text{ and }\quad\angT{\Sigma-3} = \angles{\frac{\pi}{4}}{\frac{\pi}{2}}{\frac{\pi}{4}}.
	\end{equation*}
	This assignment yields $\Delta_\Sigma$ having shape (III) as $\angT{\Sigma} = (\pi/2,\pi/4,\pi/4).$ Finally, we require $\theta_{3(\Sigma-4)} = \pi/4$. We must have the ideal tetrahedra in $\Delta_{\Sigma-4}$ having either shape (VI) or shape (III). In both cases we can set $\theta_{3(\Sigma-4)} = \pi/4$. All solutions to any remaining angle equations can be set as in \Cref{lem:angleB3}.
\end{proof}

\begin{lemma}\label{lem:onlyB2}
	Let $\tri$ be the Sakuma-Weeks triangulation of a $2$-bridge link complement associated to the word $\Omega=RL^{a_1}\cdots(L^{a_n}R\mid R^{a_n}L)$ with $a_i=2$ for all $1\leq i\leq n$. That is, the subword $\Omega' = L^{a_1}R^{a_2}\cdots(L^{a_n}\mid R^{a_n})$ consists of a single (B2) block. Then $\tri$ admits an angle structure consisting only of shapes (III) and (VII) arranged as
	\begin{center}
		\begin{tikzpicture}
			\node (VII) {(VII)};
			\node[right=of VII] (III) {\big[(III)\big]$\,^{2n-1}$};
			\node[right=of III] (VIIi) {(VII)};
			
			\draw[->] (VII) -- (III);
			\draw[->] (III) -- (VIIi);
		\end{tikzpicture}
	\end{center}
\end{lemma}

\begin{proof}
	Suppose $n=1$ so that $\Omega=RL^2R$. The angle equations describing $\tri$ are,
	\begin{align*}
		4\theta_{3(1)-2} + 2\theta_{3(1)} + 2\theta_{3(2)}     & = 2\pi, \\
		2\theta_{3(1)-1} + 2\theta_{3(2)-1} + 2\theta_{3(3)-1} & = 2\pi, \\
		\theta_{3(1)} + 2\theta_{3(2)-2} + \theta_{3(3)}       & = 2\pi, \\
		2\theta_{3(2)} + 4\theta_{3(3)-2} + 2\theta_{3(3)}     & = 2\pi.
	\end{align*}
	We assign shape (VII) to $\Delta_1$ and $\Delta_3$ as,
	\begin{equation*}
		\angT{1} = \angT{3} = \angles{\frac{\pi}{8}}{\frac{3\pi}{8}}{\frac{\pi}{2}}.
	\end{equation*}
	The remaining equations are solved by assigning shape (III) to $\Delta_2$ as \[\angT{2} = \angles{\frac{\pi}{2}}{\frac{\pi}{4}}{\frac{\pi}{4}}.\]
	
	Adding additional syllables with $a_i=2$ adds angle equations that are solved by assigning shape (III) to the added layers in the following (cyclic) order starting from $\Delta_2$,
	\begin{center}
		\begin{tikzpicture}
			\node (A) {$\angles{\dfrac{\pi}{2}}{\dfrac{\pi}{4}}{\dfrac{\pi}{4}}$};
			\node[right=of A] (B) {$\angles{\dfrac{\pi}{4}}{\dfrac{\pi}{4}}{\dfrac{\pi}{2}}$};
			\node[right=of B] (C) {$\angles{\dfrac{\pi}{4}}{\dfrac{\pi}{2}}{\dfrac{\pi}{4}}$};
			\node[right=of C] (D) {$\angles{\dfrac{\pi}{4}}{\dfrac{\pi}{4}}{\dfrac{\pi}{2}}$};
			
			\draw[-Straight Barb] (A) -- (B);
			\draw[-Straight Barb] (B) -- (C);
			\draw[-Straight Barb] (C) -- (D);
			\draw[-Straight Barb] (D) to[out=195,in=-15] (A);
		\end{tikzpicture}
	\end{center}
	The assignment of shape (VII) is as above if $\Omega$ ends with $R$, otherwise the first two angles are exchanged. 
\end{proof}

\subsection{Complexity bounds}
\label{ssec:bounds}
We now prove lower bounds on the complexity of $2$-bridge link complements through the angle structures established in the previous section. The first step of this proof is determining explicit formulae for the volume of a given block type in the decomposition of $\Omega$. The approximate volume of each of the ten shapes is given below as a multiple of the volume of a regular ideal tetrahedron $v_3$.
\begin{equation*}
	\arraycolsep=20pt\def\arraystretch{1.7}
	\begin{array}{ll}
		v_3 = \Vol(\Delta(\pi/3,\pi/3,\pi/3)) \approx 1.0149               & v_V = \Vol(\Delta(\pi/6,\; \pi/2,\; \pi/3)) \approx 0.8333v_3        \\
		v_I = \Vol(\Delta(\pi/3,\;3\pi/8,\;7\pi/24)) \approx 0.9902v_3     & v_{VI} = \Vol(\Delta(\pi/6,\; \pi/4,\; 7\pi/12)) \approx 0.7754v_3   \\
		v_{II} = \Vol(\Delta(\pi/3,\;\pi/4,\;5\pi/12)) \approx 0.9604v_3   & v_{VII} = \Vol(\Delta(\pi/8,3\pi/8,\pi/2)) \approx 0.7417v_3         \\
		v_{III} = \Vol(\Delta(\pi/4,\; \pi/4,\; \pi/2)) \approx 0.9024v_3  & v_{VIII} = \Vol(\Delta(\pi/8,\; \pi/4,\; 5\pi/8)) \approx 0.6768v_3  \\
		v_{IV} = \Vol(\Delta(5\pi/24,\;7\pi/24,\;\pi/2)) \approx 0.8855v_3 & v_{IX} = \Vol(\Delta(\pi/12,\; 7\pi/12,\; \pi/3)) \approx 0.5833v_3. \\
	\end{array}
\end{equation*}

Recall that the length of a block is the number of syllables it contains - that is, the number of maximal subwords $R^{a_i}$ or $L^{a_i}$.

Using the lemmas in the Section 5.2, we can determine the number of layers containing ideal tetrahedra of each shape type in each given block. This also provides the volume of each block type in terms of its length and is collected in \Cref{tbl:shapelayer}.
\begin{table}[htbp]
	\centering
	{\small
		\begin{tabular}{|c||c|c|c|c|c|c|c|c|c|c|c|}
			\hline
			                         & \# layers   & (0)       & (I) & (II) & (III)       & (IV) & (V)                      & (VI) & (VII)      & (VIII) & (IX) \\
			\hline\hline
			(B1)                     & $k$         & $k^\star$ & 0   & 0    & 0           & 0    & $\Delta_1,\Delta_\Sigma$ & 0    & 0          & 0      & 0    \\
			\hline
			(B2), at start           & $2k$        & 0         & 1   & 0    & $2(k-1)$    & 0    & 0                        & 1    & $\Delta_1$ & 0      & 0    \\
			\hline
			(B2), at end ($k=2$)     & 4           & 0         & 0   & 0    & 0           & 0    & 3                        & 0    & 0          & 0      & 1    \\
			\hline
			(B2), at end ($k\geq 3$) & $2k$        & 0         & 0   & 0    & $2k-5$      & 0    & 4                        & 0    & 0          & 0      & 1    \\
			\hline
			(B3)                     & $2\ell+m+1$ & 0         & 1   & $1$  & $2(\ell-1)$ & 1    & $\Delta_1$               & 1    & 0          & $m-1$  & 0    \\
			\hline
			(B3), at end             & $2\ell+m+1$ & 0         & 1   & 0    & $2\ell-1$   & 0    & $\Delta_1$               & 1    & 0          & $m$    & 0    \\
			\hline
			Unfinished (B3) ($m=1$)  & $3$         & 0         & 1   & 0    & 0           & 0    & $\Delta_1$               & 1    & 1          & 0      & 0    \\
			\hline
			Unfinished (B3) ($m>1$)  & $2\ell+m$   & 0         & 1   & 0    & $2\ell-1$   & 0    & $\Delta_1$               & 1    & 0          & $m-1$  & 0    \\
			\hline
			All (B2)                 & $2k$        & 0         & 0   & 0    & $2k-1$      & 0    & 0                        & 0    & 2          & 0      & 0    \\
			\hline 
		\end{tabular}
	}
	\caption{Number of layers containing ideal tetrahedra of each shape in each block type. The length of a block is denoted $k$ (length does not include $\Delta_1$) and the number of (B2) blocks in a (B3) block is denoted $m$ and the total length of the (B2) blocks is denoted $\ell$. Values marked by $\star$ contain one less layer of this shape if the block occurs at the end of the word $\Omega$, with the final layer replaced with the shape marked by $\Delta_\Sigma$. The entry marked by $\Delta_1$ indicates the shape of the initial layer of $\tri$ if the block occurs at the start of $\Omega$. Recall that the unfinished (B3) block occurs only when $a_{n-1}=1$ and $a_n=2$ and is always at the end of $\Omega$.}
	\label{tbl:shapelayer}
\end{table}
To prove our main result we first show that the above prescription of angle structures, combined with Thurston's lower bound on complexity \Cref{eq:full-bound}, determine that the Sakuma-Weeks triangulation is minimal if $a_i=1$ for all $1\leq i\leq n$.

\begin{lemma}[\cite{ishikawa_construction_2016}, Corollary 1.1]
	Let $K=K(\Omega)$ be the $2$-bridge link associated to the word $\Omega = RL^{a_1}\cdots(L^{a_n}R\mid R^{a_n}L)$ with $a_i=1$ for all $1\leq i\leq n$, $n\geq 1$. Let $M=S^3\backslash K$ and $\tri=\tri(M)$ be the Sakuma-Weeks triangulation of $M$, Moreover, let $c(M)$ denote the complexity of $M$. Then $\tri$ is minimal; that is 
	\begin{equation*}
		c(M) = |\tri| = 2(n+1).
	\end{equation*}
\end{lemma}
\begin{proof}
	The subword $\Omega'=L^{a_1}\cdots(L^{a_n}R\mid R^{a_n}L)$ decomposes into a single (B1) block consisting of $n$ layers. Using \Cref{tbl:shapelayer}, we note that the ideal tetrahedra in $\Delta_1$ and $\Delta_\Sigma$, where $\Sigma = n+1$, have shape (V) whilst the ideal tetrahdedra in the remaining $n-1$ layers are regular ideal. Denote the corresponding angle structure $\Theta^\star$. Combining \Cref{eq:full-bound} with \Cref{cor:volumeestimate} we obtain
	\begin{equation}\label{eq:fullcomplexitybound}
		\frac{\mathcal{V}(\Theta^\star)}{v_3} \leq c(M) \leq |\tri|.
	\end{equation}
	This gives us a lower bound on the complexity as
	
	\begin{equation}\label{eq:lowerboundminimal}
		\frac{2(2v_{V}+ (n-1)v_3)}{v_3} = 2n+1.3332. 
	\end{equation}
	Combining with the upper bound on complexity we have,
	\begin{equation*}
		2n+1.3332 \leq c(M) \leq 2n+2,
	\end{equation*}
	and conclude that $\tri$ is minimal.
\end{proof}

\begin{theorem}\label{thm:complexityBound0.8}
	Let $L=L(\Omega)$ be the $2$-bridge link associated to the word $\Omega=RL^{a_1}\cdots(L^{a_n}R\mid R^{a_n}L)$ where $a_i\in\{1,2\}$ for all $1\leq i\leq n$ and $n\geq 1$. Let $M=S^3\backslash L$ and $\tri=\tri(M)$ be the Sakuma-Weeks triangulation of $M$ and let $c(M)$ denote the complexity of $M$. Then,
	\begin{equation*}
		0.8|\tri| \leq c(M) \leq |\tri|
	\end{equation*}
\end{theorem}
Before we provide a proof we remark that the lower bound we find is conservative. For example, if $\Omega$ is such that $a_i=2$ for all $1\leq i\leq n$, then the angle structure we find gives a lower bound of approximately $0.85|\tri|\leq c(M)$ for $n=3$ and this converges to approximately $0.9|\tri|\leq c(M)$ as $n\to\infty$. Moreover, we show below that the greatest volume deficits are attained by (B3) blocks with all (B2) blocks having ``small length''. Provided that the total length of such blocks is finite, we typically see an asymptotic lower bound of approximately $0.9|\tri|\leq c(M)$. The choice of conservative bound was made as providing a complete list of constraints decreased the readability of the proof more than what was justifiable.
\begin{proof}
	Using \Cref{eq:fullcomplexitybound} it suffices to show that there is some angle structure $\Theta^\star$ on $\tri$ such that
	\begin{equation*}
		0.8\leq \frac{\mathcal{V}(\Theta^\star)}{|\tri|v_3}.
	\end{equation*}
	Consider a decomposition of $\tri$ into three subcomplexes $\mathcal{S}_1$, $\mathcal{S}_2$, and $\mathcal{S}_3$ where $\mathcal{S}_1$ contains $\Delta_1$ and the ideal tetrahedra associated to the first block in the decomposition of $\Omega$, $\mathcal{S}_3$ contains $\Delta_\Sigma$ and the ideal tetrahedra associated to the last block of the decomposition of $\Omega$, and $\mathcal{S}_2$ contains all remaining ideal tetrahedra. Denote the angle structure on $\mathcal{S}_i$ by $\Theta_i^\star$ so that $\Theta^\star$ is the concatenation of these three angle structures.
	
	We proceed by computing the volume for each $\mathcal{S}_i$ in terms of the length of the block(s) it contains different from (B1). Refer to \Cref{tbl:shapelayer} for all shapes used in what follows.
	
	The first observation we make is that the volume of a (B3) block is smallest when it contains $m$ (B2) blocks of length $1$. This follows since each (B2) block after the first replaces two ideal tetrahedra of shape (III) with two ideal tetrahedra of shape (VIII), which have a smaller volume. To minimise the volume for a fixed length of the block we maximise the number of ideal tetrahedra of shape (VIII) which is done by giving each (B2) block length 1. Consequentially, we only consider (B3) blocks of this form.
	
	Consider first $\mathcal{S}_1$. A single (B2) block consists of $2k_1$ layers plus $\Delta_1$ giving $|\mathcal{S}_1|=2(2k_1+1)$ ideal tetrahedra. Assigning the appropriate angle structures we compute
	\[
		\frac{\mathcal{V}(\Theta_1^{(B2)})}{2(2k_1+1)v_3} = \frac{2\left(v_{VII} + v_I + v_{VI} + 2(k_1-1)v_{III}\right)}{2(2k_1+1)v_3} = \frac{3.6096k_1+1.405}{2(2k_1+1)}.     
	\]
	This evaluates to $0.8357$ for $k_1=1$ and increases with $k_1$. A single (B3) block with $m_1$ (B2) blocks of length 1 consists of $3m_1+1$ layers plus $\Delta_1$ giving $|\mathcal{S}_1|=2(3m_1+2)$ ideal tetrahedra. As above we find
	\[
		\frac{\mathcal{V}(\Theta_1^{(B3)})}{2(3m_1+2)v_3} = \frac{2\left(v_V + v_I + v_{II} + v_{IV} + v_{VI} + 2(m_1-1)v_{III} + (m_1-1)v_{VIII}\right)}{2(3m_1+2)v_3} = \frac{4.9632m_1+3.9264}{2(3m_1+2)}.    
	\]
	This evaluates to $0.8889$ for $m_1=1$ and decreases with $m_1$, approaching $0.8272$ as $m_1\to\infty$.
	
	For $\mathcal{S}_2$, observe that the largest deficit for the subcomplex $\mathcal{S}_2$ must occur with a single (B3) block with $m_2$ (B2) blocks of length $1$. Adding multiple (B3) blocks -- or even (B1) blocks in between (B3) blocks -- introduces tetrahedra of shapes (0), (I), (II), (IV) and (VI), which have a larger volume than an ideal tetrahedra of shape (VIII). Hence, using the calculations above without the ideal tetrahedra of shape (V), 
	\[
		\frac{\mathcal{V}(\Theta_2^\star)}{2(3m_2+1)v_3} = \frac{4.9632m_2 + 2.2598}{2(3m_2+1)}.
	\]
	This evaluates to $0.9028$ for $m_2=1$ and, as above, decreases with $m_2$, approaching $0.8272$ as $m_2\to\infty$.
	
	Consider now $\mathcal{S}_3$. A single (B2) block consists of $2k_3$ layers for $k_3\geq 2$ giving $4k_3$ ideal tetrahedra. We compute
	\[
		\frac{\mathcal{V}(\Theta_3^{(B2)})}{4k_3v_3} = \begin{cases}
			\frac{2(3v_V+v_{IX})}{8v_3} = 0.7708,                                           & k_3=2      \\
			\frac{2(4v_V+v_{IX}+(2k_3-5)v_{III})}{4k_3v_3} = \frac{3.6096 k_3-1.191}{4k_3}, & k_3\geq 3.
		\end{cases}
	\]
	This evaluates to $0.8031$ for $k_3=3$ and increases with $k_3$. However, we need to ensure that the deficit when $k_3=2$ is balanced by the other subcomplexes. We need only check two cases. First, if $\tri = \mathcal{S}_1 \cup \mathcal{S}_3$ with $\mathcal{S}_1$ a (B3) block with $m_1$ (B2) blocks of length 1 then,
	\[
		\frac{\mathcal{V}(\Theta^\star)}{2(3m_1+6)v_3} = \frac{4.9632 m_1 + 0.0928}{2(3m_1+6)}.    
	\]
	This evaluates to $0.8364$ when $m_1=1$ and decreases with $m_1$ as above. If instead $\tri = \mathcal{S}_1\cup\mathcal{S}_2\cup\mathcal{S}_3$ with $\mathcal{S}_1$ a (B2) block of length $k_1$ and $\mathcal{S}_2$ a (B3) block with $m_2$ (B2) blocks of length 1, then we obtain
	\[
		\frac{\mathcal{V}(\Theta^\star)}{2(2k_1+3m_2+6)v_3} = \frac{3.6096 k_1 + 4.9632 m_2 + 9.8312}{2(2k_1+3m_2+6)}.    
	\]
	This evaluates to $0.8365$ for $k_1=m_2=1$ and increases with $k_1$ and decreases with $m_2$, approaching $0.8272$ as $m_2\to\infty$. 
	
	Returning to $\mathcal{S}_3$, we consider now when $\mathcal{S}_3$ is a single unfinished (B3) block with $m_3$ (B2) blocks of length 1. This gives $3m_3$ layers and hence $|\mathcal{S}_3| = 6m_3$. We compute
	\[
		\frac{\mathcal{V}(\Theta^\star)}{6m_3v_3} = \begin{cases}
			\frac{2\left(v_I+v_{VI}+v_{VII}\right)}{6v_3} = 0.8357,                                                & m_3=1      \\
			\frac{2\left(v_I+v_{VI}+(2m_3-1)v_{III}+(m_3-1)v_{VIII}\right)}{6m_3v_3} = \frac{4.9632 m+0.3728}{6m_3}, & m_3\geq 2.
		\end{cases}
	\]
	This evaluates to $0.8582$ for $m_3=2$ and decreases with $m_3$, approaching $0.8272$ as $m_3\to\infty$.
	
	For all of the above computations we have $0.8\leq\mathcal{V}(\Theta_i^\star)/|\mathcal{S}_i|v_3$ for each $1\leq i\leq 3$. Combining these subcomplexes together ensures that the inequality is true for $|\tri|$. Note that we have omitted the case when $\tri$ is a single (B3) block or a single unfinished (B3) block. Whilst this will reduce the ratios computed above, we still compute a value above $0.8$ which decreases to $0.8272$ as the number of (B2) blocks increases.
	
	Finally, we consider when $\tri$ decomposes as a single (B2) block of length $k$. We have $|\tri| = 2(2k+1)$ and compute
	\[
		\frac{\mathcal{V}(\Theta^{(B2)})}{2(2k+1)v_3} = \frac{2((2k-1)v_{III}+2v_{VII})}{2(2k+1)v_3} = \frac{3.6096k+1.162}{2(2k+1)}.
	\]
	This evaluates to $0.8381$ for $k=2$ and increases with $k$ and satisfies our desired result. However, when $k=1$ this evaluates to $0.7952$. In this case we assign a different angle structure to the triangulation. We solve the angle equations, obtained from \Cref{fig:tetrahedralayer,fig:closing-off-a}, by assigning
	\begin{equation*}
		(II)\quad \angT{1} = \angT{3} = \angles{\frac{\pi}{4}}{\frac{5\pi}{12}}{\frac{\pi}{3}}\quad\text{ and }\quad\angT{2} = \angles{\frac{2\pi}{3}}{\frac{\pi}{6}}{\frac{\pi}{6}}.
	\end{equation*}
	The volume of the new shape here is computed as $\Vol(\Delta(2\pi/3,\;\pi/6,\;\pi/6)) = 0.6666v_3$.
	Hence we have,
	\begin{equation*}
		\frac{\mathcal{V}(\Theta^\star)}{|\tri|v_3} = \frac{2(0.9604) + 0.6666}{3} = 0.8720.
	\end{equation*}
	It follows that the Sakuma-Weeks triangulation of a $2$-bridge link associated to the word $\Omega=RL^{a_1}\cdots(L^{a_n}R\mid R^{a_n}L)$ with $a_i\in\{1,2\}$ for all $1\leq i\leq n$ satisfies $0.8|\tri|\leq c(M)\leq |\tri|$.
\end{proof}
The minimal volume for each subcomplex in the preceding proof occurs with a (B3) block containing $m$ blocks of type (B2) of length one. This provides a linear relationship between $m$, the number of $a_i=2$, and the gap between certifying minimality of the Sakuma-Weeks triangulation. 
\begin{corollary}\label{cor:additivecorollary}
	Let $W = \left\{\Omega = RL^{a_1}\cdots(L^{a_n}R\mid R^{a_n}L)\mid a_i\in\{1,2\}\text{ for }1\leq i\leq n\right\}$ and for each $C\in\N$ set
	\begin{equation*}
		W_C = \left\{\Omega\in W\mid a_1+\cdots+a_n = n+C,\;n\in\N,\,n\geq C\right\}.
	\end{equation*}
	For $\Omega\in W_C$ let $K=K(\Omega)$ be the associated $2$-bridge link. Let $M=S^3\backslash K$ and $\tri=\tri(M)$ be the Sakuma-Weeks triangulation of $M$ with $|\tri| = 2(n+1+C).$ Let $c(M)$ denote the complexity of $M$. Then,
	\begin{equation*}
		2n + 1 + (0.9632C + 0.393) \leq c(M) \leq 2n + 1 + (2C + 1).
	\end{equation*}
\end{corollary}
\begin{proof}
	It follows from the proof of \Cref{thm:complexityBound0.8} that the only case we need to consider is when $\Omega$ decomposes as a single (B1) block of length $3C+1$. Let $\Phi^*$ denote the associated angle structure. The largest deficit is obtained when $\Omega$ decomposes as a single (B3) block of length $3C+1$ consisting of $C$ (B2) blocks of length $1$. Let $\Theta^*$ denote the associated angle structure. We compute the volume functional for $\Theta^*$ as 
	\begin{equation*}
		\mathcal{V}(\Theta^*) = 2\left(v_V + v_I + v_{VI} + (C-1)(2v_{III}+v_{VIII}) + v_{III} + v_{VIII}\right) \approx (4.9632C + 3.3930)v_3.
	\end{equation*}
	The largest deficit obtained from $\Omega\in W_C$ is then 
	\begin{equation*}
		\frac{\mathcal{V}(\Phi^\star)-\mathcal{V}(\Theta^\star)}{v_3} \approx \frac{2(2v_V+3C\,v_3) - (4.9632C+3.3930)v_3}{v_3} \approx 1.0368C - 0.0598.
	\end{equation*}
	This gives the lower bound on the complexity as,
	\begin{align*}
		\frac{\Vol(S^3\backslash K)}{v_3} \geq \frac{\mathcal{V}(\Theta^\star)}{v_3} & \geq \frac{\mathcal{V}(\Phi^\star)}{v_3} - (1.0369C - 0.0597) \\
		                                                                             & = 2\left(n+C-1 + \frac{2v_V}{v_3}\right) - (1.0369C - 0.0597) \\
		                                                                             & \geq 2n + 1 + (0.9632C + 0.393).
	\end{align*}
\end{proof}

In \Cref{thm:complexityBound0.8} we have determined the complexity of infinitely many hyperbolic $2$-bridge links up to a multiplicative constant of $0.8$. In the proof of this theorem we have also certified that the Sakuma-Weeks triangulation of the complement of the $2$-bridge link associated to $\Omega=RL^2R$ is minimal.

We new analyse when the general lower bound in \Cref{thm:complexityBound0.8} is better than that in \Cref{cor:additivecorollary}. For this, consider a word $\Omega\in W_C$ giving $|\tri| = 2(n+C+1)$. This gives us  
\begin{align*}
	2n+1+(0.9632C+0.393) & \leq 0.8|\tri|   \\
	                     & \leq 1.6(n+C+1) 
\end{align*}
and hence
\begin{equation}
	\label{eq:comparison}
	C \geq 0.628n - 0.325.
\end{equation}
We conclude that our additive bound is more suitable when $C$ is small relative to $n$. This aligns with the fact that \Cref{cor:additivecorollary} is meant to provide the complexity of infinitely many $2$-bridge link complements up to a fixed additive constant, only depending on the number $C$ of syllables of length two. \Cref{thm:complexityBound0.8}, on the other hand, is meant to work in the general setting, where we can capitalise on the fact that syllables of length two must become adjacent once they are numerous enough.

\subsection{Comparison using existing volume bounds}

Lower bounds for volumes of alternating links have been established by Lackenby \cite{Lackenby04Bounds} and refined by Agol, Storm and Thurston \cite{Agol07Bounds}. Work by Futer \cite[Appendix B]{gueritaud_canonical_2006} improves this lower bound multiplicatively, establishing 
$2v_3 n - 2.7066 < \Vol(S^3\backslash K(\Omega))$ for all hyperbolic $2$-bridge links constructed from a word $\Omega$ with $n$ syllables. Petronio and Vesnin \cite{petronio-two-sided} prove a further refinement of this as 
\[ v_3\,\cdot \max\{2, 2n-2.6667\} \leq \Vol(S^3\backslash K(\Omega)).\]
To the best of the authors' knowledge, this bound provides the best estimate of the volume for infinite families of hyperbolic $2$-bridge links as
\[\max\{2,2n-2.6667\}\leq c(M).\]
In particular, this bound readily provides a lower bound for the complexity of $S^3\backslash K(\Omega)$. We conclude this article by comparing our results to this lower bound on complexity.

For convenience we work in the setting of \Cref{cor:additivecorollary}. Consider $\Omega\in W_C$ for $C\geq 0$ and let $\tri$ be the Sakuma-Weeks triangulation of the complement of $K(\Omega)$. We write $|\tri| = 2(n+C+1)$. An immediate observation is that for all $n\geq 1$ and $0\leq C\leq n$ we have 
\begin{equation*}
	\max\{2,2n-2.6667\} \leq 2n+1+(0.9632C + 0.393).
\end{equation*}
That is to say incorporating the number of syllables and length of syllables in \Cref{cor:additivecorollary} provides a better lower bound on the complexity for the 2-bridge links we are considering. 

We compare the lower bound in \Cref{thm:complexityBound0.8} in two cases. First, if $n=1,2$ then the Petronio-Vesnin bound gives $2\leq c(M)$. The smallest triangulation we consider occurs when $n=1$ and $C=0$ with $0.8|\tri| = 3.2$. Hence our lower bound is better in this case. For $n\geq 3$ we consider $2n-2.6667\leq 0.8|\tri|$. This is satisfied when 
\begin{equation*}
	C \geq 0.25n - 2.6666.
\end{equation*}
Again, we can see that the bound from \Cref{thm:complexityBound0.8} works better in the case of $C$ being sufficiently large, with the case of relatively small $C$ covered by the bound from \Cref{cor:additivecorollary}.


\bibliographystyle{plain}
\bibliography{refs.bib}


\address{James Morgan\\School of Mathematics and Statistics F07, The University of Sydney, NSW 2006 Australia\\{james.morgan@sydney.edu.au\\-----}}

\address{Jonathan Spreer\\School of Mathematics and Statistics F07, The University of Sydney, NSW 2006 Australia\\{jonathan.spreer@sydney.edu.au}}

\Addresses

\end{document}